\documentclass[12pt,a4paper,reqno,dvipsnames]{amsart}

\usepackage{amsfonts}
\usepackage{amsmath}
\usepackage{amssymb,stmaryrd}
\usepackage{amsthm}
\usepackage{amsxtra}
\usepackage{bbm}
\usepackage{braket}
\usepackage{comment}
\usepackage{extarrows}
\usepackage{graphicx}
\usepackage{latexsym}
\usepackage{mathrsfs}
\usepackage{yhmath}
\usepackage{nicematrix}
\usepackage{mathtools}
\let\underbrace\LaTeXunderbrace
\let\overbrace\LaTeXoverbrace
\usepackage{float}
\usepackage{booktabs}
\usepackage{xcolor}
\usepackage{verbatim}
\usepackage{listings}
\usepackage{longtable}
\usepackage[normalem]{ulem}

\usepackage[pdfborder={0 0 0}]{hyperref} 
\hypersetup{colorlinks, citecolor=blue, linkcolor=blue, urlcolor=ForestGreen}
\hypersetup{pagebackref}
\usepackage{bm}
\usepackage{tikz}
\usepackage[initials,lite]{amsrefs} 
\AtBeginDocument{%
    \def\MR#1{}
}

\usepackage{cleveref}

\usepackage[cm]{fullpage}

\makeatother

\theoremstyle{plain}
\newtheorem{Theorem}{Theorem}[section]
\newtheorem{Lemma}[Theorem]{Lemma}
\newtheorem{Corollary}[Theorem]{Corollary}
\newtheorem{Proposition}[Theorem]{Proposition}

\theoremstyle{definition}

\newtheorem{Assumptions and Discussion}[Theorem]{Assumptions and Discussion}

\newtheorem{Example}[Theorem]{Example}
\newtheorem{Definition}[Theorem]{Definition}

\newtheorem{Remark}[Theorem]{Remark}

\newtheorem{Observation}[Theorem]{Observation}

\newtheorem{Notation}[Theorem]{Notation}

\newtheorem{Setting}[Theorem]{Setting}
\theoremstyle{remark}

\newtheorem*{acknowledgment*}{Acknowledgment}

\usepackage[shortlabels]{enumitem}
\SetEnumerateShortLabel{a}{\textup{(\alph*)}}
\SetEnumerateShortLabel{A}{\textup{(\Alph*)}}
\SetEnumerateShortLabel{1}{\textup{(\arabic*)}}
\SetEnumerateShortLabel{i}{\textup{(\roman*)}}
\SetEnumerateShortLabel{I}{\textup{(\Roman*)}}

\def\lex{\operatorname{lex}}

\def\ceil#1{\left\lceil #1 \right\rceil}

\def\codim{\operatorname{codim}}

\def\dim{\operatorname{dim}}

\def\diff{\operatorname{diff}}

\def\floor#1{\left\lfloor #1 \right\rfloor}

\def\Ht{\operatorname{ht}} 

\def\indeg{\operatorname{indeg}} 

\def\ini{\operatorname{in}} 

\def\isom{\cong}

\def\KK{{\mathbb K}}

\def\lex{{\operatorname{lex}}}

\def\MC{\operatorname{MC}}

\def\Mon{\operatorname{Mon}} 

\def\NN{{\mathbb N}}

\def\part{\operatorname{part}}

\def\sgn{\operatorname{sgn}}

\def\sJ{\mathscr{J}}

\def\sort{\operatorname{sort}}

\def\Sss{\operatorname{Sss}}

\def\ZZ{{\mathbb Z}}

\newcommand\bdalpha{{\bm \alpha}}

\newcommand\bda{{\bm a}}

\newcommand\bdb{{\bm b}}

\newcommand\bdc{{\bm c}}

\newcommand\bdd{{\bm d}}

\newcommand\bdeta{{\bm \eta}}

\def\bdeta{{\bm \eta}}

\newcommand\bds{{\bm s}}
\newcommand\bdT{{\bm T}}

\newcommand\bdtau{{\bm \tau}}

\newcommand\bdx{{\bm x}}

\newcommand\calA{\mathcal{A}}

\newcommand\calD{\mathcal{D}}
\newcommand\calE{\mathcal{E}}
\newcommand\calF{\mathcal{F}}
\newcommand\calG{\mathcal{G}}

\newcommand\fraka{\mathfrak{a}}

\def\frakS{\mathfrak{S}}

\newcommand\length{\operatorname{length}}

\newcommand{\pd}{\operatorname{pd}}

\newcommand{\rank}{\operatorname{rank}}
\def\reg{\operatorname{reg}}

\begin{document}

\title{Regularity and multiplicity of Veronese type algebras}

\author[Kuei-Nuan Lin, Yi-Huang Shen]{Kuei-Nuan Lin and Yi-Huang Shen}

\thanks{2020 {\em Mathematics Subject Classification}.
    Primary 05E40 
    13F55  	
    13F65  	
    Secondary 14M25  	
    13H15  	
    13D02  
}

\thanks{Keyword: Algebra of Veronese type, Regularity, Cohen--Macaulay, Multiplicity}

\address{Department of Mathematics, The Penn State University, McKeesport, PA, 15132, USA}
\email{linkn@psu.edu}

\address{School of Mathematical Sciences, University of Science and Technology of China, Hefei, Anhui, 230026, P.R.~China}
\email{yhshen@ustc.edu.cn}

\begin{abstract}
    In this paper, we study the algebra of Veronese type. We show that the presentation ideal of this algebra has an initial ideal whose Alexander dual has linear quotients. As an application, we explicitly obtain the Castelnuovo-Mumford regularity of the Veronese type algebra. Furthermore, we give an effective upper bound on the multiplicity of this algebra.
\end{abstract}

\maketitle

\section{Introduction}

Let $S=\KK[x_{1},\dots,x_{n}]$ be a polynomial ring over a field $\KK$ with $n\ge 2$.  Suppose that $I$ is an ideal minimally generated by some monomials $f_1,\ldots, f_u$ in $S$. The \emph{semigroup ring} associated to $I$, denoted by $\KK[I]$, 
is the subalgebra of $S$ generated by $f_1,\dots,f_u$. This ring is
also known as the \emph{toric ring} associated with $I$.  
In this work, we focus on investigating several algebraic invariants of $\KK[I]$ associated with an ideal of Veronese type.

We fix a degree $d$ and a sequence $\bdalpha=(\alpha_1,\ldots, \alpha_n)$ of integers with $1\le \alpha_1 \le \cdots \leq \alpha_n\le d$ and $d <\sum_{i=1}^n \alpha_i$.  Let $\NN=\Set{0,1,2,\dots}$ be the set of non-negative integers and
\[
    I_{d,\bdalpha}\coloneqq \Braket{\bdx^\bdc: \bdc=(c_1,\dots,c_n)\in \NN^n,\, \sum_{i=1}^n c_i =d \text{ and } c_i\le \alpha_i \text{ for each $i$ }}
\]
be \emph{an ideal of Veronese type}. Related, let $\calA_{d,\bdalpha}=\KK[I_{d,\bdalpha}]\subset S$ and call it an \emph{algebra of Veronese type}. If $\alpha_1=\cdots=\alpha_n$, then $I_{d,\bdalpha}$ is a special strongly symmetric shifted ideal. Moreover, if $\alpha_i=d$ for all $i$, then $\calA_{d,\bdalpha}$ is the $d$-{th} Veronese subring of $S$.   Algebras of Veronese type have been studied from 
various viewpoints, see \cites{MR627814,MR524782,Sturmfels}. De Negri and Hibi also classified the Gorensteinness of $\calA_{d,\bdalpha}$ in \cite[Theorem 2.4]{MR1458806}. In addition, Costantini and Seceleanu studied algebras associated with strongly symmetric shifted ideals in \cite{arXiv:2208.10476}. In this paper, we continue their work and compute the Castelnuovo--Mumford regularity and the multiplicity of $\calA_{d,\bdalpha}$.

Recall that for a finitely generated 
graded module $M$ over the polynomial ring 
$S$,
the \emph{Castelnuovo--Mumford regularity} of $M$, denoted by $\reg(M)$, is $\max\Set{j-i : \beta_{i,j}\ne 0}$, where the $\beta_{i,j}$'s are the graded Betti numbers of $M$. This regularity can be used to bound the degree of the syzygy generators of the module $M$. Another important invariant that we study here is the \emph{multiplicity} $\mathtt{e}(M)$ of $M=\bigoplus_k M_k$ with respect to the graded maximal ideal,
where $M_k$ is the degree $k$ component of the graded module $M$.
Recall that the Hilbert function $H(M,k):=\dim_{\KK}(M_k)$ is eventually a polynomial in $k$ of degree $\dim(M)-1$. The leading coefficient of this polynomial is of form $\mathtt{e}(M)/(\dim(M)-1)!$. 
If $M$ is the homogeneous coordinate ring of a projective variety $X$, then the multiplicity is just the \emph{degree} of $X$. 

One way to study the semigroup ring $\KK[I]$ algebraically is to investigate its presentation ideal.
Let $\KK[\bdT]= \KK[T_{f_1}, \ldots, T_{f_u}]$ be a polynomial ring in $u$ variables over $\KK$.
Then the \emph{presentation ideal} $J$ is the kernel of the canonical $\KK$-algebra homomorphism $\psi: \KK[\bdT]\rightarrow \KK[I]$ with $\psi(T_{f_i})=f_i$. Obviously,  $J$ is a prime ideal and we have $\KK[\bdT]/J \cong \KK[I]$.

The study of algebra $\KK[I]$ becomes more feasible if such a Gr\"obner basis of the presentation ideal $J$ is known; see \cite{MR3595300} for the case of chordal bipartite graphs, \cite{MR2955237} for the $d$-{th} Veronese subring, and \cite{Lin-Shen3DMulti} for the case of three-dimensional Ferrers diagrams. 
This is because the invariants of the initial ideal and the original ideal are closely related, c.f.~ \cite[Section 3.3]{MR2724673}. According to the work of Conca and Varbaro in \cite{CV}, these relations become tight when the initial ideal is squarefree.  
On the other hand, without knowing the Gr\"obner basis of $J$, finding the Castelnuovo--Mumford regularity or the multiplicity of $\KK[I]$, for example, becomes a difficult task; cf.~\cite{MR1953712}.

It is worth noting that if $J$ has a quadratic Gr\"obner basis, then the algebra $\KK[I]$ is Koszul (that is, $\KK$ has a linear resolution over $\KK[I]$) by \cite{Froberg}.
The authors of the current paper proved in \cite{Lin-Shen3D} that if $I$ is associated with a three-dimensional Ferrers diagram satisfying a mild property, then the presentation ideal $J$ has a quadratic Gr\"obner basis. This is the foundation for the calculation of the Castelnuovo--Mumford regularity and the multiplicity in \cite{Lin-Shen3DMulti}. However, the way we obtained the quadratic Gr\"obner basis seems to be model-dependent. 
A more useful tool is the idea of \emph{sorting monomials} introduced by Sturmfels. With this technique, 
Sturmfels \cite{Sturmfels} proved that the presentation ideal of $\calA_{d,\bdalpha}$ has a quadratic Gr\"obner basis.  
More recently, Herzog et al.~\cite{arXiv:2209.10799} generalized the notion of sortability to the non-equigenerated case and proved that the presentation ideal of the algebra associated with a sortable monomial ideal also has a quadratic Gr\"obner basis. 

With the quadratic Gr\"obner basis of $J$ described earlier  by Sturmfels, we can study the Veronese type algebra $\calA_{d,\bdalpha}$ in more detail. As a main result, we obtain the Castelnuovo--Mumford regularity of $\calA_{d,\bdalpha}$ explicitly in \Cref{regSUM}. This generalizes the result  
\cite[Corollary 2.12, Remark 2.13]{BVV} of Bruns, Vasconcelos and Villarreal and \cite[Theorem 4.2]{MR2955237} of Nitsche
in the case of the $d$-{th} Veronese subring.
We then offer an effective bound on the multiplicity of $\calA_{d,\bdalpha}$ in \Cref{MultiBound}. 
Notice that, the multiplicity of $\calA_{d,\bdalpha}$ is actually the normalized volume of the Newton polytope of the ideal $I_{d,\bdalpha}$, as it is an Ehrhart ring by \cite[Corollary 4.10]{EVY}.

This work is organized as follows. In Section \ref{Basic}, 
we recall some essential definitions and terminology that we will need later.
In Section \ref{maxCliques}, we use the combinatorial structure of the 
sorting operation to study the algebra $\calA_{d,\bdalpha}$. From the initial ideal of the presentation ideal $J$ of $\calA_{d,\bdalpha}$, we introduce a structural graph $\calG_{d,\bdalpha}$ and study its maximal cliques. As a result, we derive the generators of the Alexander dual of $\ini(J)$ in \Cref{cor:Delta_maximal_clique}. In Section \ref{Sec:Shelling}, we propose a carefully constructed order in \Cref{Order}.  With it, we show in \Cref{prop:linearQuotient} that this Alexander dual ideal has linear quotients.  Combinatorial details of the quotients are also presented in this section. Finally, in Section \ref{Sec:app}, we gather all the tools and results and present the two main theorems mentioned above.

\section{Preliminaries}
\label{Basic}
Throughout this paper, we fix a positive integer $n\ge 2$. Following the convention, $[n]$ stands for the set $\{1,2,\dots,n\}$. Let $S=\KK[x_{1},\dots,x_{n}]$ be a polynomial ring over a field $\KK$, and $\frakS_n$ be the symmetric group of $[n]$.

\begin{Notation}
    \begin{enumerate}[a]
        \item Let $\bda\in \ZZ^n$ be a tuple written in boldface. Unless otherwise stated, we usually write $a_i$ correspondingly with subscript, not in boldface, for its $i$-{th} coordinate. Namely, we expect that $\bda=(a_1,\dots,a_n)\in \ZZ^n$. Following this convention, if we encounter several tuples $\bda^1,\dots,\bda^t\in \ZZ^n$, we will have $\bda^i=(a_1^i,\dots,a_n^i)$ for each $i\in [t]$.
        \item Suppose that $\bda\in \NN^n$. We will write $\bdx^\bda$ for the monomial $x_1^{a_1}\cdots x_n^{a_n}$ in $S$ 
            and define $|\bda|\coloneqq a_1+\cdots +a_n$. 
        \item Let $I$ be a monomial ideal of $S$. We write $G(I)$ for the minimal monomial generating set
            of $I$, and $\mu(I)= \#G(I)$ for the minimal number of generators. 
    \end{enumerate}
\end{Notation}

Recall that Sturmfels introduced in \cite[Chapter 14]{Sturmfels} a {sorting operator} for monomials of the same degree.

\begin{Definition}
    Let $V_{n,d}\coloneqq \{\bda\in \NN^n:|\bda|=d\}$.
    \begin{enumerate}[i]
        \item Let $\Mon_d$ be the set of monomials of degree $d$ in $S=\KK[x_1,\dots,x_n]$. Then, we have the \emph{sorting operator} for $p\ge 2$:
            \[
                \sort: \underbrace{\Mon_d\times \cdots \times \Mon_d}_{\text{$p$ times}} \to \underbrace{\Mon_d\times \cdots \times \Mon_d}_{\text{$p$ times}},\qquad (u_1,\dots,u_p)\mapsto (v_1,\dots,v_p),
            \]
            which is defined as follows. Suppose that $u_1\cdots u_p=x_{i_1}x_{i_2}\cdots x_{i_{pd}}$ with $i_1\le i_2\le \cdots\le i_{pd}$. Then, $v_k\coloneqq x_{i_{k}}x_{i_{p+k}}\cdots x_{i_{(d-1)p+k}}$ for $k\in [p]$.  The sequence $(u_1,\dots,u_p)$ of monomials in $\Mon_d$ is called \emph{sorted} if $\sort(u_1,\dots,u_p)=(u_1,\dots,u_p)$. A subset $U$ of $V_{n,d}$ is \emph{sortable} if $\sort(U\times U)\subseteq U\times U$.
        \item  Since we have a natural correspondence between $\bdx^{\bda}\in \Mon_d$ and $\bda\in V_{n,d}$, by abuse of notation, we also derive a sorting operator
            \[
                \sort:V_{n,d}\times \cdots \times V_{n,d} \to V_{n,d} \times
                \cdots \times V_{n,d}.
            \]
            Similarly, we can also talk about whether a subset of $V_{n,d}$ is sortable.
    \end{enumerate}
\end{Definition}

In this paper, we always assume the following.

\begin{Setting}
    \label{set:ideal_veronese}
    Suppose that $n\ge 3$. 
    We fix a degree $d$ and a sequence $\bdalpha=(\alpha_1,\ldots, \alpha_n)$ of integers with $1\le \alpha_1 \le \cdots \leq \alpha_n\le d$ and $d <
    |\bdalpha|
    $. Let $V_{n,d}^{\bdalpha}\coloneqq \Set{\bdc=(c_1,\dots,c_n)\in V_{n,d}:c_i\le \alpha_i \text{ for all $i$}}$. Furthermore, let $I_{d,\bdalpha}\coloneqq \braket{\bdx^{\bdc}:\bdc\in V_{n,d}^{\bdalpha}}$ be the Veronese type ideal in $S$, and $\calA_{d,\bdalpha}=\mathbb{K}[I_{d,\bdalpha}]$ be the Veronese type algebra. 
\end{Setting}

\begin{Remark}
    \label{prop:analytic_spread}
    \begin{enumerate}[a]
        \item When $n=2$, it is not difficult to see that the Veronese type algebra $\calA_{d,\bdalpha}$ is isomorphic to a Veronese ring, up to a shift. 
        \item If $d=|\bdalpha|$, then the Veronese type algebra $\calA_{d,\bdalpha}$ is generated principally by the monomial $\bdx^{\bdalpha}$.
        \item  The Krull dimension of $\calA_{d,\bdalpha}$ is $n$ by \cite[Corollary 2.2(4)]{MR1458806}.

    \end{enumerate}  
\end{Remark}

\begin{Remark}
    \label{thm:definingIdeals} 
    Let $\calA_{d,\bdalpha}$ be the Veronese type algebra in
    \Cref{set:ideal_veronese}.
    Since $V_{n,d}^{\bdalpha}$ is sortable by
    \cite[Proposition 6.11]{MR2850142}, the set
    \[
        \sJ \coloneqq \Set{\underline{T_{\bda^1}
            T_{\bda^2}}-T_{\bdb^1}T_{\bdb^2} :\bda^1,\bda^2\in
            V_{n,d}^{\bdalpha}, \text{ $(\bda^1,\bda^2)$ unsorted and
        $\sort(\bda^1,\bda^2)=(\bdb^1,\bdb^2)$}} 
    \]
    is a Gr\"obner basis of the presentation ideal $J$ of $\calA_{d,\bdalpha}$ with respect to a term order such that the marked \textup{(}underlined\textup{)} monomials are the leading terms by \cite[Theorem 6.16]{MR2850142}. In particular, $J$ has a quadratic squarefree initial ideal. Consequently, $\calA_{d,\bdalpha}$ is a Cohen--Macaulay Koszul normal domain, by 
    \cite[Corollary 2.2]{MR1458806}, \cite[Theorems 5.16, 5.17, and 6.7]{MR2850142}, and \cite[
    Corollary 4.10]{EVY}.
\end{Remark}

\section{Maximal cliques}
\label{maxCliques}

Let $\calA_{d,\bdalpha}$ be the Veronese type algebra in \Cref{set:ideal_veronese}.  In one of the main results of this paper, we determine the Castelnuovo--Mumford regularity of the algebra $\calA_{d,\bdalpha}$.
Recall that if $\beta_{i,j}$ is the graded Betti number of $\calA_{d,\bdalpha}$ considered
as a $\KK[\bdT]$-module,
then the Castelnuovo--Mumford regularity of $\calA_{d,\bdalpha}$ is defined to be
\[
    \reg(\calA_{d,\bdalpha})\coloneqq \max_{i,j} \Set{j-i: \beta_{i,j}\ne 0}.
\]
Our exploration of the Castelnuovo--Mumford regularity depends on the
following key observations.

\begin{Observation}
    \label{regToPd}
    Let $J$ be the presentation ideal of $\calA_{d,\bdalpha}$ generated by the set $\sJ$ in \Cref{thm:definingIdeals}. Since $\calA_{d,\bdalpha}$ is Cohen--Macaulay and $\ini(J)$ is squarefree by \Cref{thm:definingIdeals}, it follows from \cite[Corollary 2.7]{CV} that the quotient ring $\KK[\bdT]/\ini(J)$ is Cohen--Macaulay.
    Let $(\ini(J))^\vee$ be the Alexander dual of the squarefree ideal $\ini(J)$. By \cite[Corollary 2.7]{CV}
    and \cite[Proposition 8.1.10]{MR2724673}, $\reg(\calA_{d,\bdalpha})=\reg(\KK[\bdT]/J) = \reg(\KK[\bdT]/\ini(J))=\pd((\ini(J))^\vee)$.
\end{Observation}

Inspired by the above observation, we turn to find the projective dimension of $(\ini(J))^\vee$.
To do this, we proceed as follows.
First, since $\KK[\bdT]/\ini(J)$ is Cohen--Macaulay, $(\ini(J))^\vee$ is an equigenerated ideal.
We will describe the minimal monomial generators of this ideal as the maximal cliques of length $n$ of some graph in \Cref{prop:max_clique_size}. 
Then, in Section \ref{Sec:Shelling}, we give a total order on the minimal monomial generating set $G((\ini(J))^\vee)$ and show that $(\ini(J))^\vee$ has linear quotients with respect to this order. Details of these quotients are important by the following lemma.

\begin{Lemma}
    [{\cite[Corollary 8.2.2]{MR2724673}}]
    \label{MaxLenghtToPd}
    Let $I$ be an equigenerated ideal in the polynomial ring $S$. Suppose that the minimal generating set $G(I)$ consists of $f_1,\dots,f_m$ such that $Q_i\coloneqq \braket{f_1,\dots,f_{i-1}}:_S f_i$ is linear for each $i$. Let $p=\max_i \mu(Q_i)$ be the maximum of the minimal numbers of generators of the successive linear quotient ideals. Then, the top nonzero total Betti number of $I$ is $\beta_p(I)$, which has the value $\#\Set{i:\mu(Q_i)=p}$. In particular, the projective dimension of $I$ is given by $p$.
\end{Lemma}

\begin{Definition}
    \label{eta}
    \begin{enumerate}[a]
        \item Recall that $V_{n,d}$ is the set $\{\bda\in \NN^n:|\bda|=d\}$. For any distinct $\bda, \bdb \in V_{n,d}$, we write $\bda >_{\lex} \bdb$ if the leftmost nonzero entry of $\bda-\bdb$ is  positive.

        \item Let $\bdeta = (\eta_1,\dots,\eta_n)$ be the smallest tuple in $V_{n,d}^{\bdalpha}$ with respect to $>_{\lex}$, namely, if $\bda \in V_{n,d}^{\bdalpha}$, then $\bda \ge_{\lex} \bdeta$. It is clear that $\eta_n=\alpha_n$ and $\eta_i=\min\{\alpha_i,d-\eta_n-\eta_{n-1}-\cdots-\eta_{i+1}\}$ for $i=n-1,n-2,\dots,1$.

        \item Let $\calG=\calG(d,\bdalpha)$ be a simple graph on the set $V_{n,d}^{\bdalpha}$, such that $\{\bda,\bdb\}$ with $\bda>_{\lex}\bdb$ is an edge if and only if $(\bda,\bdb)$ is sorted. A \emph{clique} of $\calG$ is a collection of vertices, in which, every two different vertices are adjacent. A \emph{maximal clique} is a clique that is maximal with respect to inclusion. Let $\MC(\calG)=\MC(\calG(d,\bdalpha))$ be the set of maximal cliques of $\calG$.
    \end{enumerate}
\end{Definition}

\begin{Remark}
    \label{rmk:maximal_cliques}
    \begin{enumerate}[a]
        \item If $(\bda,\bdb)$ is a sorted pair in $V_{n,d}$, then $\bda\ge_{\lex}\bdb$.
        \item \label{rmk:maximal_cliques_a}
            The fact (ii) before \cite[Theorem 14.2]{Sturmfels} can be restated as follows. Suppose that $\{\bda^1>_{\lex} \bda^2>_{\lex} \cdots >_{\lex}\bda^s\}$ is a subset of $V_{n,d}^{\bdalpha}$ such that $\sort(\bda^i,\bda^j)=(\bda^i,\bda^j)$ for all $1\le i<j\le s$. Then, $\sort(\bda^1,\bda^2,\dots, \bda^s) =(\bda^1, \bda^2, \dots, \bda^s)$. In other words, sorted subsets of $V_{n,d}^{\bdalpha}$ are precisely cliques in $\calG$.

        \item \label{rmk:maximal_cliques_b}
            It follows from the description in \Cref{thm:definingIdeals} and the fact just stated in \ref{rmk:maximal_cliques_a} that the squarefree monomial ideal $\ini(J)$ is the edge ideal of the complement graph $\calG^{\complement}$. Suppose that $T_{\bda^1}T_{\bda^2}\cdots T_{\bda^{m}}\in G(\ini(J)^{\vee})$ is a minimal monomial generator of the Alexander dual ideal. This is equivalent to saying that $\{\bda^1, \dots, \bda^m\}$ is a minimal vertex cover of the complement graph $\calG^{\complement}$, by \cite[Corollary 9.1.5]{MR2724673}. In other words, the set complement $\{\bda^1, \dots, \bda^{m}\}^{\complement}$ is a maximal clique of $\calG$.
    \end{enumerate}
\end{Remark}

Consequently, in this section, we turn to the study of the maximal cliques.

\begin{Proposition}
    \label{prop:max_clique_size}
    Every maximal clique of $\calG$ has cardinality $n$.
\end{Proposition}

\begin{proof}
    The Veronese type algebra $\calA_{d,\bdalpha}$ is Cohen--Macaulay by \Cref{thm:definingIdeals}. In other words, the presentation ideal $J$ is Cohen--Macaulay. Since $\ini(J)$ is squarefree, this implies that $\ini(J)$ is Cohen--Macaulay by \cite[Corollary 2.7]{CV}. In particular, $\ini(J)$ is height unmixed. Notice that $\Ht(\ini(J))=\Ht(J)=\#\bdT-n$ by \Cref{prop:analytic_spread}, where $\#\bdT$ is the number of variables in $\KK[\bdT]$, i.e., the dimension of this polynomial ring. 
    Consequently, every minimal monomial generator of $\ini(J)^{\vee}$ has degree $\#\bdT -n$, and every maximal clique has size $\#\bdT-(\#\bdT -n)=n$ by \Cref{rmk:maximal_cliques} \ref{rmk:maximal_cliques_b}.
\end{proof}

We need more tools to describe these maximal cliques in detail.

\begin{Definition}
    Let $\bda,\bdb\in\ZZ^n$ be two tuples. Suppose that the following
    condition is satisfied.
    \begin{enumerate}
        \item [(\texttt{S}):]
            There exist $1\le i_1<i_2<\cdots<i_{2k-1}<i_{2k}\le n$ such that
            $b_{i_{2j-1}}=a_{i_{2j-1}}-1$ while $b_{i_{2j}}=a_{i_{2j}}+1$ for
            $1\le j\le k$. Furthermore, for $t\in [n]\setminus
            \{i_1,i_2,\dots,i_{2k}\}$, one has $b_{t}=a_t$.
    \end{enumerate}
    In this case, the \emph{sorting signature set} $\Delta(\bda,\bdb)$ of this
    pair is defined to be
    \[
        \Delta(\bda,\bdb)\coloneqq \bigcup_{j=1}^k [i_{2j-1},i_{2j}),
    \]
    which is considered as the union of intervals on the real line.
    Consequently, the \emph{length} of this set is
    $\sum_{j=1}^k(i_{2j}-i_{2j-1})$, denoted by $\length(\Delta(\bda,\bdb))$.
\end{Definition}

From the definition, we can directly verify the following simple fact.

\begin{Remark}
    \label{rmk:sort_output}
    Let $\bda$ and $\bdb$ be two tuples in $V_{n,d}$, and $P\coloneqq\{i\mid a_i\not\equiv b_i\mod 2\}$. Since $\left|\bda\right|=\left|\bdb\right|$, the cardinality of $P$ is even. Whence, we may assume that $P=\{i_1<i_2<\cdots<i_{2k}\}$.
Suppose that $\sort(\bda,\bdb)=(\bdc,\bdd)$. One can verify, using the definition of the sorting operator, that 
\[
    c_i=d_i=\frac{a_i+b_i}{2}
\]
for $i\notin P$. Furthermore,
\[
    c_{i_{2j-1}}=\frac{a_{2j-1}+b_{2j-1}+1}{2}, \quad
    d_{i_{2j-1}}=\frac{a_{2j-1}+b_{2j-1}-1}{2}, \quad
    c_{i_{2j}}=\frac{a_{2j}+b_{2j}-1}{2}, \quad
    d_{i_{2j}}=\frac{a_{2j}+b_{2j}+1}{2},
\]
for $1\le j\le k$.
\end{Remark}

\begin{Lemma}
    \label{lem:3_1}
    Let $\bda>_{\lex}\bdb$ be two tuples in $V_{n,d}$. Then,
    $\sort(\bda,\bdb)=(\bda,\bdb)$ if and only if the condition
    \textup{(\texttt{S})} is satisfied.
\end{Lemma}
\begin{proof}
    This follows directly from the observation in \Cref{rmk:sort_output}.
\end{proof}

The following are further important observations when using the sorting operator.

\begin{Lemma}
    \label{lem:sss_separate}
    If $\{\bda>_{\lex}\bdb>_{\lex}\bdc\}$ is a sorted subset in $V_{n,d}$,
    then $\Delta(\bda,\bdc)=\Delta(\bda,\bdb)\sqcup \Delta(\bdb,\bdc)$.
\end{Lemma}

\begin{proof}
    It suffices to show that $\Delta(\bda,\bdb) \cap \Delta(\bdb,\bdc) =\emptyset$. Suppose for contradiction that $\Delta(\bda,\bdb) \cap \Delta(\bdb,\bdc)$ is not empty.  Then, we can find suitable intersecting $[i_{2j-1},i_{2j})$ and $[i_{2j'-1}', i_{2j'}')$ in the definition of $\Delta(\bda,\bdb)$ and $\Delta(\bdb, \bdc)$ respectively. Let $j$ and $j'$ be the smallest such that this happens.
    \begin{enumerate}[a]
        \item Suppose that $i_{2j-1}=i_{2j'-1}'$. Whence $b_{i_{2j-1}}=a_{i_{2j-1}}-1$ and $c_{i_{2j-1}}=b_{i_{2j-1}}-1$.  This forces $c_{i_{2j-1}}=a_{i_{2j-1}}-2$. Since condition \textup{(\texttt{S})} is not satisfied here, we have a contradiction to the assumption that $(\bda,\bdc)$ is a sorted pair by \Cref{lem:3_1}.
        \item \label{lem:sss_separate_b}
            Suppose that $i_{2j-1}<i_{2j'-1}'<i_{2j}$. Now, we derive $i_{2j'-2}'\le i_{2j-1}$ from the minimality assumption.  From the interval positions, it is clear that $a_{i_{2j'-1}'}=b_{i_{2j'-1}'}=c_{i_{2j'-1}'}+1$ and consequently $a_{i_{2j'-1}'} +b_{i_{2j'-1}'} +c_{i_{2j'-1}'}\equiv 2\pmod 3$. Meanwhile, it is not difficult to check that $-1+\sum_{k\le i_{2j-1}}a_k=\sum_{k\le i_{2j-1}}b_k=\sum_{k\le i_{2j-1}}c_k$. Furthermore, we have $a_t=b_t=c_t$ for $i_{2j-1}<t<i_{2j'-1}'$.  Whence, the sorting operator will produce $a_{i_{2j'-1}'}+1 =b_{i_{2j'-1}'}=c_{i_{2j'-1}'}$, which is a contradiction.
        \item Suppose that $i_{2j'-1}'<i_{2j-1}<i_{2j'}'$. This case is similar to \ref{lem:sss_separate_b} above. \qedhere
    \end{enumerate}
\end{proof}

\begin{Lemma}
    \label{lem:3_4}
    \begin{enumerate}[a]
        \item If $\{\bda>_{\lex}\bdb>_{\lex}\bdd\}$ and $\{\bdb>_{\lex}\bdc>_{\lex}\bdd\}$ are two cliques in $\calG$, then $\{\bda>_{\lex}\bdb>_{\lex}\bdc >_{\lex}\bdd\}$ is also a clique in $\calG$.
        \item Symmetrically, if $\{\bda>_{\lex}\bdb>_{\lex}\bdc\}$ and $\{\bda>_{\lex}\bdc>_{\lex}\bdd\}$  are two cliques in $\calG$, then $\{\bda>_{\lex}\bdb>_{\lex}\bdc >_{\lex}\bdd\}$ is also a clique in $\calG$.
    \end{enumerate}
\end{Lemma}

\begin{proof}
    By symmetry, we will only consider the first case. Since $\Delta(\bda,\bdb)\cap\Delta(\bdb,\bdd)=\emptyset$ while $\Delta(\bdb,\bdc)\subseteq \Delta(\bdb,\bdd)$ by \Cref{lem:sss_separate}, we have $\Delta(\bda,\bdb)\cap\Delta(\bdb,\bdc)=\emptyset$. Then, we can apply \Cref{lem:3_1} and \Cref{lem:sss_separate} to show that $\sort(\bda,\bdc)=(\bda,\bdc)$ with $\Delta(\bda,\bdc) =\Delta(\bda,\bdb)\sqcup \Delta(\bdb,\bdc)$. Whence, $\{\bda>_{\lex}\bdb>_{\lex}\bdc >_{\lex}\bdd\}$ is a clique in $\calG$.
\end{proof}

We know from \Cref{prop:max_clique_size} that every maximal clique has cardinality $n$. Additional information can be drawn here. 

\begin{Corollary}
    \label{cor:Delta_maximal_clique}
    For each maximal clique $\{\bda^1>_{\lex}\cdots>_{\lex}\bda^n\}$ in
    $\calG$, we have $\length(\Delta(\bda^i,\bda^j))=j-i$ when $1\le i < j \le
    n$. Furthermore, $\Delta(\bda^1,\bda^n) = [1,n)$. In other words, $a_1^n=
    a_1^1-1$, $a_n^n=a_n^1+1$ and $a_i^n=a_i^1$ for $i\in \{2,\dots, n-1\}$.
\end{Corollary}
\begin{proof}
    Since $\Delta(\bda^1,\bda^n)\subseteq [1,n)$ and
    $\length(\Delta(\bda^1,\bda^n))\ge \sum_{i=1}^{n-1} \length(\Delta(\bda^i,
    \bda^{i+1}))\ge n-1$ by \Cref{lem:sss_separate}, the first two statements are clear.
    The last statement follows from the definition of the sorting signature
    set.
\end{proof}

\begin{Corollary}
    \label{cor:maximal_clique_shift_right}
    Let $\{\bda^1>_{\lex}\bda^2>_{\lex}\cdots >_{\lex}\bda^n
    >_{\lex}\bda^{n+1}\}$ be a set of vertices in $\calG$. Suppose that
    $\{\bda^1>_{\lex}\bda^2 >_{\lex} \cdots >_{\lex}\bda^n\}$ is a maximal
    clique of $\calG$ and $\Delta(\bda^1,\bda^2)=\Delta(\bda^n,\bda^{n+1})$.
    Then, $\{\bda^2>_{\lex}\cdots >_{\lex}\bda^n >_{\lex} \bda^{n+1}\}$ is
    also a maximal clique of $\calG$.
\end{Corollary}
\begin{proof}
    We have $\Delta(\bda^2,\bda^n)\sqcup
    \Delta(\bda^n,\bda^{n+1})=\Delta(\bda^2,\bda^n)\sqcup
    \Delta(\bda^1,\bda^{2})=\Delta(\bda^1,\bda^n)=[1,n)$ by
    \Cref{cor:Delta_maximal_clique}. Thus, we can verify by definition that
    $\Delta(\bda^2,\bda^{n+1})=[1,n)$. In particular,
    $\{\bda^2>_{\lex}\bda^{n+1}\}$ is an edge of $\calG$ by \Cref{lem:3_1}.
    Now, as $\{\bda^2>_{\lex} \cdots >_{\lex} \bda^n\}$ is a clique and
    $\{\bda^n >_{\lex} \bda^{n+1}\}$ is an edge, $\{\bda^2>_{\lex}\cdots
    >_{\lex} \bda^n>_{\lex}\bda^{n+1}\}$ is also a maximal clique of $\calG$
    by \Cref{lem:3_4}.
\end{proof}

\section{Linear Quotients}
\label{Sec:Shelling}

In this section, we continue to assume that $\calA_{d,\bdalpha}$ is the Veronese type algebra of \Cref{set:ideal_veronese} and $\bdeta$ is the tuple given in \Cref{eta}.  As explained at the beginning of Section \ref{maxCliques}, we intend to show that the Alexander dual ideal $(\ini(J))^\vee$ has linear quotients. This tactic allows us to have more control over its minimal free resolution. In particular, we can explicitly calculate the Castelnuovo--Mumford regularity of $\calA_{d,\bdalpha}$ in \Cref{Sec:app}, and also give a reasonable upper bound on its multiplicity. 

To prove the linear quotient property, we need to impose a total order $\prec$ on the minimal monomial generating set $G(\ini(J)^\vee)$, such that with respect to this order, the ideal $\ini(J)^\vee$ has linear quotients.  Let $\MC(\calG)$ be the set of maximal cliques of the graph $\calG=\calG(d,\bdalpha)$. As observed in \Cref{rmk:maximal_cliques}\ref{rmk:maximal_cliques_b}, 
\[
    G(\ini(J)^\vee) = \Set{\bdT_{A^\complement}\coloneqq \prod_{v\in
    V_{n,d}^{\bdalpha}\setminus A}T_v : A\in \MC(\calG)}.
\]
Thus, we will consider the corresponding total order, still denoted by
$\prec$, on $\MC(\calG)$. By definition, we want to show that the quotient
ideal $\braket{\bdT_{B^\complement} : B\prec A}: \bdT_{A^\complement}$ is
generated by some ring variables of $\KK[\bdT]$ for each $A\in \MC(\calG)$.
Notice that
\begin{equation}
    \label{translate_lq_sets}
    \braket{\bdT_{B^\complement}}:\bdT_{A^\complement}=
    \prod_{v\in A\setminus B}T_v.
\end{equation}
Therefore, for any given maximal clique $B$ with $B\prec A$, it suffices to
find a suitable maximal clique $D$ with $D\prec A$ such that $A\setminus D$ is
a singleton set with $A\setminus D \subseteq A\setminus B$.

Unfortunately, the total order $\prec$ introduced here is really involved. We have to postpone its debut to \Cref{Order}. Before that, we need to make some preparations.

\begin{Definition}
    \label{def:ranks}
    \begin{enumerate}[a]
        \item A priori, a maximal clique of the graph $\calG$ is a set of vertices.  When we write a maximal clique $A=(\bda^1, \bda^2, \dots, \bda^n)$ in the tuple form, we intend to indicate that $\bda^1>_\lex \bda^2>_\lex \cdots >_\lex \bda^n$.

        \item Two maximal cliques $A=(\bda^1,\bda^2,\dots,\bda^n)$ and $B=(\bdb^1,\dots,\bdb^n)$ are called \emph{equivalent}, if and only if $\bda^1=\bdb^1$. It follows from \Cref{cor:Delta_maximal_clique} that this is also equivalent to saying that $\bda^n=\bdb^n$. With respect to this binary relation, we have \emph{equivalence classes}. We will write $\calE_A$ for the equivalence class to which $A$ belongs.

        \item \label{item:ranks_c}
            The \emph{rank} of a tuple $\bda=(a_1,\dots,a_n)\in V_{n,d}$ 
            (with respect to the given tuple $\bdeta$)
            is defined to be
            \[
                \rank(\bda)\coloneqq \sum_{j=1}^{n-1}(a_j-\eta_j)(n-j).
            \]
            It is clear that $\rank(\bdeta)=0$. Furthermore, if
            $\bda>_{\lex}\bdb$ belong to some common maximal clique of
            $\calG$, then it is easy to verify directly that
            $\rank(\bda)= \rank(\bdb)+ \length(\Delta(\bda,\bdb))$.

        \item For each maximal clique $A=(\bda^1,\dots,\bda^n)$, we define the
            \emph{rank} of $\calE_{A}$ to be $\rank(\bda^n)$. This is
            well-defined, since if $B=(\bdb^1,\dots,\bdb^n)$ and $A$ are
            equivalent, then $\bdb^n=\bda^n$.

        \item Suppose that $A=(\bda^1,\dots,\bda^n)$, where 
            $\bda^i=(a_1^i,\dots,a_n^i)\in \ZZ^n$ and $\sum_ka_k^i=d$ for each $i\in [n]$.
            Assume that there exists a permutation $(s_1,\dots,s_{n-1})$ of $\{1,2,\dots,n-1\}$, written in one-line notation, such that $\Delta(\bda^i,\bda^{i+1})=[s_i,s_i+1)$ for $i\in [n-1]$. Then we say the tuple $(s_1,\dots,s_{n-1})$ is the \emph{signature} of $A$ and denote it by $\sgn(A)$. Of course, we are mostly interested in the case where $A$ is a maximal clique of $\calG$. Whence, by \Cref{lem:sss_separate} and \Cref{cor:Delta_maximal_clique}, the signature of $A$ must exist.

        \item Let $A=(\bda^1,\dots,\bda^n)$ be a maximal clique. If there
            exists a maximal clique $B=(\bdb^1,\dots,\bdb^n)$ such that
            $(\bda^2,\dots,\bda^n)=(\bdb^1,\dots,\bdb^{n-1})$, then we will
            say that $B$ is the \emph{root} of $A$ and write $B=rA$. It
            follows from \Cref{cor:Delta_maximal_clique} that if the root of
            $A$ exists, then it is unique.
    \end{enumerate}
\end{Definition}

\begin{Remark}
    \begin{enumerate}[a]
        \item For any $\bda\in V_{n,d}^\bdalpha$, if $\bda\ne \bdeta$, we can find $\bdb\in V_{n,d}^\bdalpha$ such that $\bda>_{\lex}\bdb\ge_{\lex} \bdeta$, $\sort(\bda,\bdb)=(\bda,\bdb)$, and $\length(\bda,\bdb)=1$. If we use this repeatedly, we can find by induction $\bdb^0,\bdb^1,\dots,\bdb^m\in V_{n,d}^\bdalpha$ such that $\bdb^0=\bda$, $\bdb^m=\bdeta$ and $\rank(\bdb^i)=\rank(\bdb^{i+1})+1$ for each $i$. In particular, $\rank(\bda)=m\in \NN$.

        \item For each maximal clique $A$, we have $\rank(A)\in \NN$. 

        \item If $B=rA$, then $\rank(A)=\rank(B)+1$.

        \item We have precisely one equivalence class $\calE_A$ such that $\rank(\calE_A)=0$. If $B=(\bdb^1,\dots,\bdb^n)$ belongs to such $\calE_A$, then $\bdb^n=\bdeta$.
    \end{enumerate}
\end{Remark}

Here is the converse of \Cref{cor:Delta_maximal_clique}.

\begin{Lemma}
    \label{lem:legal_initial_tuple}
    Let $\bda^1=(a_1^1,\dots,a_n^1)$ be a tuple in $V_{n,d}^{\bdalpha}$.
    Suppose that
    $1 \le a_1^1$ while $a_n^1 \le \alpha_n-1$. Then, there exists a maximal
    clique in $\calG$ of the form $\Set{\bda^1>_{\lex} \cdots
    >_{\lex}\bda^n}$.
\end{Lemma}
\begin{proof}
    It is clear that $\bda^n \coloneqq (a_1^1-1, a_2^1, \dots, a_{n-1}^1, a_n^1+1)$ belongs to $V_{n,d}^{\bdalpha}$, by assumption. Notice that $\{\bda^1>_{\lex}\bda^n\}$ is a clique in $\calG$ and $\Delta(\bda^1, \bda^n)=[1,n)$. Thus, we can complete $\{\bda^1>_{\lex} \bda^n\}$ to a maximal clique. This maximal clique must have the form $\Set{\bda^1>_{\lex} \cdots >_{\lex}\bda^n}$, due to \Cref{prop:max_clique_size} and the rank reason given at the end of \Cref{def:ranks}\ref{item:ranks_c}.
\end{proof}

Next, we describe a necessary and sufficient condition for a given maximal clique $A$ with $\rank(A) > 0$ and such that $rA$ does not exist. This characterization ensures that we pick the correct element when ordering for the linear quotients.

\begin{Lemma}
    \label{lem:rA_not_exists}
    Suppose that $A=(\bda^1,\dots,\bda^n)$ is a maximal clique such that $\bda^1=(a_1^1,\dots,a_n^1)$ and $\sgn(A)=(s_1,\dots,s_{n-1})$. If $\rank(A)>0$, then $rA$ does not exist, if and only if $a_1^1=1$ and $s_1=1$, or $a_n^1=\alpha_n-1$ and $s_1=n-1$.
\end{Lemma}
\begin{proof}
    Suppose that $\bda^2=(a_1^2,\dots,a_n^2)$. Let $\bda^{n+1}$ be the tuple such that $\Delta(\bda^2,\bda^{n+1})=[1,n)$. It is easy to see that $\bda^{n+1}=(a_1^2-1,a_2^2,\dots,a_{n-1}^2,a_n^2+1)$ and $\Delta(\bda^n,\bda^{n+1})= [s_1,s_1+1)$.  Then, $rA$ does not exist if and only if $\bda^2,\dots,\bda^n,\bda^{n+1}$ is not a legitimate maximal clique. By \Cref{cor:maximal_clique_shift_right}, 
    the latter happens precisely when $\bda^{n+1}\notin V_{n,d}^{\bdalpha}$, which means either $a_1^2=0$ or $a_n^2=\alpha_n$.
    \begin{enumerate}[a]
        \item Suppose that $a_1^2=0$. We claim that $s_1=1$. Otherwise, $s_{j}=1$ for some $2\le j\le n-1$. This implies that $\bda^k=(a_1^2,\dots)$ for $k=1,\dots,j$, and $\bda^k=(a_1^2-1,\dots)$ for $k=j+1,\dots,n$. But as $a_1^2-1=-1$ in this case, we have a contradiction. Now, since $s_1=1$, it is clear that $a_1^1=1$.
        \item Suppose that $a_n^2=\alpha_n$. The argument is similar.
    \end{enumerate}

    Conversely, if $a_1^1=1$ and $s_1=1$, or $a_n^1=\alpha_n-1$ and $s_1=n-1$,
    then $a_1^2=0$ or $a_n^2=\alpha_n$. By the argument at the beginning of
    this proof, $rA$ does not exist.
\end{proof}

To facilitate exhibiting the claimed linear quotients property, we need the
subsequent handy tool.

\begin{Lemma}
    \label{one_different}
    Let $A=(\bda^1,\ldots,\bda^n)$ and $B=(\bdb^1,\dots,\bdb^n)$ be two tuples of elements in $\ZZ^n$ such that $\bda^1=\bdb^1$ and $\bda^n=\bdb^n$. Suppose that $\sgn(A)= (s_1,\ldots,s_{n-1})$ 
    and $\sgn(B)=(t_1,\ldots,t_{n-1})$. 
    Then, the following conditions are equivalent for an index $k\in [n-2]$:
    \begin{enumerate}[a]
        \item \label{oneSwitch}
            the signature $\sgn(B)$ takes the form $(s_1, \ldots,s_{k-1},
            s_{k+1},s_{k}, s_{k+2},\dots,s_{n-1})$;
        \item \label{DiffOne}
            the difference set $\diff(A,B)\coloneqq \{i: \bda^i \ne
            \bdb^{i}\}$ is exactly $\{k+1\}$.
    \end{enumerate}
    If in addition $A$ and $B$ are two maximal cliques in an equivalence class $\calE$, then the following is an additional equivalent condition:
    \begin{enumerate}[resume*]
        \item \label{QuotientOne}
            the quotient ideal $\braket{\bdT_{B^\complement}} :_{\KK[\bdT]}
            \bdT_{A^\complement}$ is generated by $T_{\bda^{k+1}}$.
    \end{enumerate}
\end{Lemma}

\begin{proof}
    Firstly, we show that \ref{oneSwitch} $\Rightarrow$ \ref{DiffOne}. By assumption, we have $\bda^1=\bdb^1$ and $s_i=t_i$ for $i=1,2,\dots,k-1$. Since
    \[
        \Delta(\bda^{i},\bda^{i+1})=[s_i,s_i+1)=
        [t_i,t_i+1)=\Delta(\bdb^{i},\bdb^{i+1}),
    \] 
    by induction, we have $\bda^j=\bdb^j$ for $j=1,2,\dots,k$. Similarly, we have $\bda^n=\bdb^n$ and $s_i=t_i$ for $i=n-1,n-2,\dots,k+2$ by assumption. Thus, we also have $\bda^j=\bdb^j$ for $j=n,n-1,\dots,k+2$. On the other hand, since $s_k\ne t_k$ while $\bda^k=\bdb^k$, we have
    \[
        \Delta(\bda^{k},\bda^{k+1})=[s_k,s_k+1)\ne [t_k,t_k+1) = \Delta(\bdb^{k},\bdb^{k+1}),
    \]
    and hence $\bda^{k+1}\neq \bdb^{k+1}$. It is now clear that $\diff(A,B)=\{k+1\}$.

    Secondly, we show that \ref{DiffOne} $\Rightarrow$ \ref{oneSwitch}. By assumption, we have $\bda^j=\bdb^j$ whenever $j\ne k+1$. 
    Since
    \[
        \bigsqcup_{j\le i}\,[s_j,s_j+1) = \Delta(\bda^1,\bda^{i+1}) =
        \Delta(\bdb^1,\bdb^{i+1}) = \bigsqcup_{j\le i}\,[t_j,t_j+1)
    \]
    for all $1\le i \le k-1$, we must have $s_i=t_i$ for each such $i$. Similarly, we have $s_i=t_i$ when $k+2\le i \le n-1$, by looking at $\Delta(\bda^i,\bda^{n})=\Delta(\bdb^i,\bdb^n)$. Notice that $\{s_1,\dots,s_{n-1}\}=\{1,\dots,n-1\}=\{t_1,\dots,t_{n-1}\}$. Therefore, $\{s_k,s_{k+1}\}=\{t_k,t_{k+1}\}$. Since $A\ne B$, there is only one possibility for this: $s_k=t_{k+1}$ and $s_{k+1}=t_k$.

    Finally, assume that $A$ and $B$ are two maximal cliques in $\calE$. The equivalence of \ref{DiffOne} and \ref{QuotientOne} is then clear from \Cref{translate_lq_sets}.
\end{proof}

We need tools to determine whether a potential maximal clique is really legitimate.

\begin{Definition}
    Suppose that $\bda=(a_1,\dots,a_n)\in V_{n,d}^{\bdalpha}$. If there exists some $\bdb\in V_{n,d}^{\bdalpha}$ such that $\Delta(\bda,\bdb) =[s,s+1)$, we say that we apply an \emph{$s$-jump} to $\bda$ in order to get $\bdb$. It is clear that such an operation exists if and only if $a_s>0$ and $a_{s+1}<\alpha_{s+1}$.
\end{Definition}

\begin{Remark}
    \label{rmk:legitimate_jump}
    Suppose that $A=(\bda^1,\dots,\bda^n)$ is a maximal clique and $\sgn(A)= (s_1, \dots, s_{n-1})$. For each $i,j\in [n]$, we have $a_j^i \in \{a_j^1-1, a_j^1,a_j^1+1\}$. The case $a_j^i=a_j^1-1$ occurs if and only if the $j$-jump is applied before $\bda^i$ (i.e., $j=s_k$ for some $k<i$), while the $(j-1)$-jump is not applied before it. Similarly, the case $a_j^i= a_j^1+1$ happens if and only if the $(j-1)$-jump is applied before $\bda^i$, while the $j$-jump is not applied before it.
\end{Remark}

\begin{Definition}
    \label{def:partial_order}
    Let $\calE$ be an equivalence class, in which, every maximal clique starts with $\bda^1=(a^1_1, \dots, a^1_n)$. We consider a partial order $\triangleleft$ on $[n-1]$ with respect to $\calE$ as follows. Suppose that $i\in [n-1]$.
    \begin{itemize}
        \item If $a_{i}^1=0$ and $i>1$, then we require that \underline{$(i-1) \, \triangleleft \, i$}.
        \item Likewise, if $a_{i+1}^1=\alpha_{i+1}$ and $i+1<n$, then we require that \underline{$(i+1) \, \triangleleft \, i$}.
    \end{itemize}
    After considering each $i\in [n-1]$, the underlined parts \emph{generate} a partial order $\triangleleft$ on $[n-1]$.  The induced poset will be called the \emph{poset of obstructions} with respect to $\calE$.
\end{Definition}

The next two lemmas justify the terminology of this poset from the point of view of allowing permutations in $\frakS_{n-1}$ to be legitimate signatures.

\begin{Lemma}
    \label{lem:po}
    The poset of \Cref{def:partial_order} is well-defined. Furthermore, suppose that $A=(\bda^1,\dots,\bda^n)$ is a maximal clique in $\calE$ and $\sgn(A) =(s_1, \dots,s_{n-1})$. Then, the following condition is satisfied:
    \begin{enumerate}
        \item [\textup{(\texttt{PO}):}]
            whenever $s_i\,\triangleleft\, s_j$, we have $i<j$.
    \end{enumerate}
\end{Lemma}
\begin{proof}
    For each $i\in [n-1]$, by definition, we apply an $s_i$-jump to $\bda^i$ in order to get $\bda^{i+1}$.
    \begin{itemize}
        \item If $a_{s_i}^1=0$ and $s_i>1$, then we must have applied an $(s_i-1)$-jump somewhere before the $s_i$-jump. Meanwhile, we require that $(s_{i}-1) \, \triangleleft \, s_i$ in \Cref{def:partial_order}.

        \item Likewise, if $a_{s_i+1}^1=\alpha_{s_i+1}$ and $s_i+1<n$, then we must have applied an $(s_i+1)$-jump somewhere before the $s_i$-jump.  Meanwhile, we require that $(s_{i}+1)\, \triangleleft \, s_i$ in \Cref{def:partial_order}.
    \end{itemize}
    Note that we won't have $s_i \, \triangleleft \,(s_i+1)$ and $(s_i+1)\, \triangleleft \, s_i$ at the same time: the first requires the $s_i$-jump to be applied before the $(s_i+1)$-jump in $A$, while the second requires the opposite. This shows that the poset is well-defined. It is also clear that the condition (\texttt{PO}) holds. 
\end{proof}

\begin{Lemma}
    \label{rmk:legal_signature}
    Conversely, let $\bds=(s_1,\dots,s_{n-1})$ be a permutation of in $\frakS_{n-1}$. Suppose that the condition \textup{(\texttt{PO})} is satisfied by $\bds$.  Then, $\bds$ is a legitimate signature with respect to $\calE$, namely, there exists some $A\in \calE$ such that $\sgn(A)=\bds$.
\end{Lemma}
\begin{proof}
    Suppose that the maximal cliques in $\calE$ all start with $\bda^1$ and end with $\bda^n$. It suffices to construct $\bda^2,\dots,\bda^{n-1}\in V_{n,d}^{\bdalpha}$, such that $\Delta(\bda^{i-1},\bda^{i}) =[s_{i-1}, s_{i-1}+1)$ when $i\ge 2$. We do this by induction on $i$.
    The degenerate base case when $i=1$ is trivial.
    Next, assume that $\bda^1,\dots,\bda^{i-1}$ have been constructed and $2\le i\le n-1$.
    \begin{enumerate}[a]
        \item If $a_{s_{i-1}}^{i-1}>0$ and $a_{s_{i-1}+1}^{i-1}< \alpha_{s_{i-1}+1}$, it is clear that we can apply the $s_{i-1}$-jump to $\bda^{i-1}$ to get a tuple $\bda^{i}\in V_{n,d}^{\bdalpha}$. Certainly $\Delta(\bda^{i-1},\bda^i)= [s_{i-1},s_{i-1}+1)$. 
        \item Suppose that $a_{s_{i-1}}^{i-1}= 0$. As $a_{s_{i-1}}^1\ge 0$, we deduce that $a_{s_{i-1}}^{i-1}\ne a_{s_{i-1}}^1+1$. Since $\bds$ is a permutation, we cannot apply the $s_{i-1}$-jump before $\bda^{i-1}$. Thus, it follows from \Cref{rmk:legitimate_jump} that we cannot use the $(s_{i-1}-1)$-jump before $\bda^{i-1}$. 
            \begin{enumerate}[i]
                \item Suppose that $s_{i-1}=1$. Whence, $s_k\ne 1$ for $k< i-1$. Consequently, we deduce that $a_{1}^1=\cdots = a_{1}^{i-1} = 0$. But this is impossible since it is necessary that $a_1^1\ge 1$.

                \item Suppose that $s_{i-1}> 1$. If in addition $a_{s_{i-1}}^1= 0$, we have $(s_{i-1}-1)\, \triangleleft \, s_{i-1}$ from \Cref{def:partial_order}. Then, the condition (\texttt{PO}) implies that the $(s_{i-1}-1)$-jump is applied before the $s_{i-1}$-jump, i.e., the $(s_{i-1}-1)$-jump is applied before $\bda^{i-1}$, a contradiction. If instead $a_{s_{i-1}}^1> 0$, then $a_{s_{i-1}}^{i-1}= a_{s_{i-1}}^1-1$. It follows from \Cref{rmk:legitimate_jump} that the $s_{i-1}$-jump is applied before $\bda^{i-1}$, again a contradiction.
            \end{enumerate}

        \item Suppose that $a_{s_{i-1}+1}^{i-1}= \alpha_{s_{i-1}+1}$. We can argue as in the previous case to see that this is impossible.
            \qedhere
    \end{enumerate}
\end{proof}

The final step is to provide each equivalence class with additional structures.

\begin{Remark}
    \label{rmk:cont_compare}
    Suppose that $i\,\triangleleft\, j$ in the poset of obstructions with respect to  $\calE$. If $i<j$, then we have $i\,\triangleleft\, (i+1) \,\triangleleft\, \cdots \,\triangleleft\,(j-1)\,\triangleleft\, j$. If instead $i>j$, then we have $i\,\triangleleft\, (i-1)\,\triangleleft\, \cdots \,\triangleleft\,(j+1)\,\triangleleft\, j$.
\end{Remark}

\begin{Definition}
    \label{def:L}
    Let $\calE$ be an equivalence class of maximal cliques.
    \begin{enumerate}[i]
        \item In $\calE$, we will take a maximal clique $L=L(\calE)$ as follows. When $a_1^1>1$, we assume that $\kappa_1=0$.  If instead $a_1^1=1$, then let $\kappa_1$ be the largest integer such that $\kappa_1\le n-1$ and $a_2^1=\cdots=a_{\kappa_1}^1=0$.  When $a_2^1>0$, we will simply take $\kappa_1=1$.  Symmetrically, if $a_n^1=\alpha_n-1$, then let $\kappa_2$ be the smallest integer such that $\kappa_2\ge 2$ and $a_j^1=\alpha_j$ for all $\kappa_2\leq j\leq n-1$.  When $a_{n-1}^1<\alpha_{n-1}$, we will simply take $\kappa_2=n$.  It is clear that $\kappa_1<\kappa_2$ if both exist. Furthermore, in \Cref{def:partial_order}, we have the relations $1 \, \triangleleft \, 2 \, \triangleleft \, \cdots \, \triangleleft \, \kappa_1$ and $(n-1)\, \triangleleft\, (n-2) \, \triangleleft \, \cdots \, \triangleleft \, (\kappa_2-1)$.  Note that these relations are the only nontrivial ones that include the integers $1,2,\dots,\kappa_1-1,\kappa_2,\kappa_2+1,\dots,n-1$ in the poset of obstructions with respect to $\calE$. On the other hand, although it is possible to have like $(\kappa_1+1)\,\triangleleft\, \kappa_1$ or $(\kappa_2-2)\,\triangleleft\,(\kappa_2-1)$ in the poset, we won't have any $t\in [n-1]$ such that either $\kappa_1 \,\triangleleft\, t$ or $(\kappa_2-1)\, \triangleleft\, t$, by the choice of $\kappa_1$ and $\kappa_2$, as well as \Cref{rmk:cont_compare}. Thus, we can choose an $L$ in $\calE$ such that
            \[
                \qquad \qquad 
                \sgn(L) \coloneqq
                \begin{cases}
                    (\tau_1, \dots, \tau_{\kappa_2-\kappa_1-2}, \underbrace{n-1, n-2,
                    \dots, \kappa_2-1}, \underbrace{1, 2, \dots, \kappa_1}), & \text{if
                    $\kappa_1<\kappa_2-1$,} \\[1em]
                    (\underbrace{n-1, n-2, \dots, \kappa_2}, \underbrace{1, 2,
                    \dots, \kappa_1}), & \text{if $\kappa_1=\kappa_2-1$}
                \end{cases}
            \]
            for suitable $\tau_i$'s that are compatible with the poset of obstructions. In the case where $\kappa_2$ is not defined, namely when $a_n^1<\alpha_n-1$, we can superficially assume that $\kappa_2=n+1$ and treat it as a degenerate case of the first one.   
            A priori, the $\sgn(L)$ just constructed is only a permutation in $\frakS_{n-1}$. It is indeed a legal signature since such a maximal clique $L$ exists by \Cref{rmk:legal_signature}.

        \item We will call $(\calE,L)$ a \emph{marked equivalence class}. This is an equivalence class with a chosen representative. Since we will take this particular $L$ for this $\calE$ once and for all in the rest of this paper, we will simply write $(\calE,L)$ as $\calE$.

        \item Suppose that $\bdtau\coloneqq \sgn(L)= (\tau_1, \tau_2, \ldots, \tau_{n-1})$.
            Now, take any $A\in \calE$ and assume that $\sgn(A)=(k_1,\dots,k_{n-1})$. Let $s_i$ be the index such that $\tau_{s_i}=k_i$. We will say that the \emph{relative signature of $A$ with respect to $\bdtau$} (or equivalently, with respect to $L$) is $\sgn_{\bdtau}(A)\coloneqq (s_1,\dots,s_{n-1})$. Obviously, $\sgn_\bdtau(L)=(1,2,\dots,n-1)$.
    \end{enumerate}
\end{Definition}

Since we have all the necessary definitions and notations in place, we are
ready to state the order that gives rise to the linear quotient property.
Recall that $\MC(\calG)$ is the set of maximal cliques of $\calG$.

\begin{Setting}
    [\textbf{Rules of order}]
    \label{Order}
    Let $\prec$ be a total order on $\MC(\calG)$ satisfying the following conditions.
    \begin{enumerate}[a]
        \item If $\rank(B)<\rank(A)$, then $B\prec A$.
        \item Suppose that $\rank(B)=\rank(A)$, $\calE_A\ne \calE_B$, and $B\prec A$. Then, for any $A'\in \calE_A$ and $B'\in \calE_B$, we have $B'\prec A'$. Therefore, $\prec$ also induces a total order on the set of equivalence classes.
        \item \label{lem:lq_cross_groups}
            Let $\calE$ be an equivalence class and $L$ be the special maximal clique we chose in \Cref{def:L}.
            Suppose that $\bdtau\coloneqq \sgn(L)=(\tau_1,\dots,\tau_{n-1})$. The restriction of $\prec$ to $\calE$ is given  by the lexicographical order with respect to $\tau_1<\tau_2<\cdots <\tau_{n-1}$. In other words, if $A,B \in \calE$, then $B\prec A$ if and only if the first nonzero entry of $\sgn_{\bdtau}(B)-\sgn_{\bdtau}(A)$ is negative.
    \end{enumerate}
    Such a total order $\prec$ exists, but in general, it is not unique.
    In the rest of this paper, we will simply fix one that works.
\end{Setting}

\begin{Example}
    \label{exam:one_max_clique}
    We present here an example showing how the maximal cliques in a marked equivalence class $(\calE,L)$ are ordered according to \Cref{Order}. Consider $n=5$, $d=8$, and $\bdalpha=(2,2,2,3,3)$. Let $\calE$ be the equivalence containing
    \[
        A=((1, 1, 1, 3, 2), (0, 2, 1, 3, 2), (0, 1, 2, 3, 2), (0, 1, 2, 2, 3),
        (0, 1, 1, 3, 3)).
    \]
    Write $A=(\bda^1,\dots,\bda^5)$. Note that $rA$ does not exist. This is because if $\Delta(\bda^2,\bda^{5+1})=[1,5)$, then $\bda^{5+1}= (-1,2,1,3,3)$, which does not belong to $V_{n,d}^{\bdalpha}$. Using the notation in \Cref{def:L}, we have $\kappa_1=1$ and $\kappa_2=4$. Consequently, we choose a maximal clique $L$ with $\bdtau\coloneqq \sgn(L)=(2, 4, 3, 1)$. Since $\bda^1$ is given, we have
    \[
        L=((1, 1, 1, 3, 2), (1, 0, 2, 3, 2), (1, 0, 2,
        2, 3), (1, 0, 1, 3, 3), (0, 1, 1, 3, 3))
    \]
    for our equivalence class $\calE$.

    The poset of obstructions, defined in \Cref{def:partial_order}, contains only the nontrivial relation $4\, \triangleleft\, 3$.  As a result, we have exactly $4!/2=12$ maximal cliques in the equivalence class $\calE$. In 
    \Cref{tabb}
    we list all their signatures and their relative signatures with respect to $\bdtau$.
    \begin{table}[htbp]
        \caption{Ordered maximal cliques in an equivalence class}
        \label{tabb}
        \begin{minipage}[t]{0.25\linewidth}
            \renewcommand{\arraystretch}{1.2}
            \centering
            {\small
                \begin{tabular}{ccc}
                    \hline
                    & $\sgn(A)$ & $\sgn_{\bdtau}(A)$
                    \tabularnewline
                    \hline
                    $1$ & $(2,4,3,1)$ & $(1,2,3,4)$
                    \tabularnewline
                    $2$ & $(2,4,1,3)$ & $(1,2,4,3)$
                    \tabularnewline
                    $3$ & $(2,1,4,3)$ & $(1,4,2,3)$
                    \tabularnewline
                    $4$ & $(4,2,3,1)$ & $(2,1,3,4)$
                    \tabularnewline
                    $5$ & $(4,2,1,3)$ & $(2,1,4,3)$
                    \tabularnewline
                    $6$ & $(4,3,2,1)$ & $(2,3,1,4)$
                    \tabularnewline
                    \hline
            \end{tabular}}
        \end{minipage}
        \quad
        \begin{minipage}[t]{0.25\linewidth}
            \renewcommand{\arraystretch}{1.2}
            \centering
            {\small
                \begin{tabular}{ccc}
                    \hline
                    & $\sgn(A)$ & $\sgn_{\bdtau}(A)$
                    \tabularnewline
                    \hline
                    $7$ & $(4,3,1,2)$ & $(2,3,4,1)$
                    \tabularnewline
                    $8$ & $(4,1,2,3)$ & $(2,4,1,3)$
                    \tabularnewline
                    $9$ & $(4,1,3,2)$ & $(2,4,3,1)$
                    \tabularnewline
                    $10$ & $(1,2,4,3)$ & $(4,1,2,3)$
                    \tabularnewline
                    $11$ & $(1,4,2,3)$ & $(4,2,1,3)$
                    \tabularnewline
                    $12$ & $(1,4,3,2)$ & $(4,2,3,1)$
                    \tabularnewline
                    \hline
            \end{tabular}}
        \end{minipage}
    \end{table}
    Due to lack of space, we will not explicitly list every maximal clique in this equivalence class.  We only mention one here for illustration.  Since every such maximal clique starts with $\bda^1=(1, 1, 1, 3, 2)$, the maximal clique $B$ satisfying
    $\sgn(B)=(1, 4, 2, 3)$ is
    \[
        B=((1, 1, 1, 3, 2), (0, 2, 1, 3, 2), (0, 2, 1, 2, 3), (0, 1, 2, 2, 3),
        (0, 1, 1, 3, 3)).
    \]
    We can then check that $\sgn_{\bdtau}(B)=(4,2,1,3)$.
\end{Example}

The rest of this section is devoted to showing the linear quotient property with respect to the order $\prec$ introduced in \Cref{Order}. Let $A$ be a maximal clique. Given \Cref{translate_lq_sets}, to show that the quotient ideal $\braket{\bdT_{B^\complement} : B\prec A}: \bdT_{A^\complement}$ is linear, 
we show that for every maximal clique $B\prec A$, we can find $C\prec A$ such that $\#\diff(A,C)=1$ and $\diff(A,C)\subseteq \diff(A,B)$.
The following technical lemma constructs the candidate maximal cliques that we will need in the later proof.

\begin{Lemma}
    \label{lem:exchange}
    Suppose that $A=(\bda^1,\dots,\bda^n)$ and $B=(\bdb^1,\dots,\bdb^n)$ are two different maximal cliques in the same equivalence class $\calE$ such that $\sgn(A)=(s_1,\dots,s_{n-1})$ and $\sgn(B)= (t_1, \dots,t_{n-1})$. Suppose that $s_i=t_i$ for $i\in \{1,2, \dots, k\} \cup \{\ell,\ell+1,\dots,n-1\}$ and $s_{k+1}\ne t_{k+1}$; we allow $\ell=n$ so that the part $\{\ell, \ell+1,\dots, n-1\}$ disappears. Then, we can find maximal cliques $C,D\in \calE$ such that
    \[
        \sgn(C)= (s_1, \dots, s_k, s_{k+1}, t_{k+1}, p_{k+3},\dots, p_{\ell-1},
        s_\ell, \dots, s_{n-1})
    \]
    and
    \[
        \sgn(D) =(s_1, \dots, s_k, t_{k+1}, s_{k+1},
        q_{k+3},\dots, q_{\ell-1}, s_\ell, \dots,s_{n-1})
    \]
    for suitable $p_i$'s and $q_i$'s.
\end{Lemma}
\begin{proof}
    Suppose that $\bda^{k+1}=(a_1^{k+1},\dots,a_n^{k+1})$. Then
    \[
        \bda^{k+2}=(a_1^{k+1},\dots,a_{s_{k+1}-1}^{k+1},
        a_{s_{k+1}}^{k+1}-1,a_{s_{k+1}+1}^{k+1}+1,a_{s_{k+1}+2}^{k+1},\dots,a_n^{k+1}).
    \]
    Since $\bda^1=\bdb^1$ and $s_i=t_i$ for $i\in [k]$, we have
    $\bdb^{k+1}=\bda^{k+1}$ and consequently
    \[
        \bdb^{k+2}=(a_1^{k+1},\dots,a_{t_{k+1}-1}^{k+1},
        a_{t_{k+1}}^{k+1}-1,a_{t_{k+1}+1}^{k+1}+1,
        a_{t_{k+1}+2}^{k+1},\dots,a_n^{k+1}).
    \]
    Without loss of generality, we can assume that $s_{k+1}<t_{k+1}$.
    \begin{enumerate}[a]
        \item Suppose that $s_{k+1}+1=t_{k+1}$. Let $\bdc^{k+3}$ be the tuple such that $\Delta(\bda^{k+2},\bdc^{k+3})=[t_{k+1},t_{k+1}+1)$, namely,
            \[
                \bdc^{k+3}=(a_1^{k+1},\dots,a_{s_{k+1}-1}^{k+1},
                a_{s_{k+1}}^{k+1}-1, a_{s_{k+1}+1}^{k+1},
                a_{s_{k+1}+2}^{k+1}+1,a_{s_{k+1}+3}^{k+1}, \dots,a_n^{k+1}).
            \]
            Due to the existence of $\bda^{k+2}$ and $\bdb^{k+2}$, it can be verified that $\bdc^{k+3}\in V_{n,d}^{\bdalpha}$. Furthermore,
            $\Delta(\bdc^{k+3}, \bda^\ell)=\Delta(\bda^{k+1},\bda^{\ell})
            \setminus [s_{k+1}, s_{k+1}+2)$. Whence, $\bda^{k+2}>_{\lex}
            \bdc^{k+3} >_{\lex}\bda^\ell$ is a clique in $\calG$. It is then
            not difficult to see that $\bda^1>_{\lex} \dots >_{\lex}\bda^{k+2}
            >_{\lex}\bdc^{k+3}>_{\lex} \bda^\ell>_{\lex}\cdots>_{\lex}\bda^n$
            is a clique, by \Cref{lem:3_4}. As $\Delta(\bda^1, \dots, \bda^n)
            =[1,n)$, we can complete it to a maximal clique, which must have 
            the form $\bda^1>_{\lex}\dots>_{\lex}\bda^{k+2}>_{\lex} \bdc^{k+3}
            >_{\lex}\bdc^{k+4} >_{\lex}>\cdots> \bdc^{\ell-1} >_{\lex}
            \bda^\ell >_{\lex} \cdots>_{\lex}\bda^n$ by
            \Cref{cor:Delta_maximal_clique}. We will take this as our expected
            maximal clique $C$. Now, $\bda^1>_{\lex}\dots>_{\lex} \bda^{k+1}
            >_{\lex}\bdb^{k+2}>_{\lex} \bdc^{k+3}>_{\lex}\bdc^{k+4} >_{\lex}
            \cdots >\bdc^{\ell-1}>_{\lex}\bda^\ell >_{\lex}\cdots >_{\lex}
            \bda^n$ is also a maximal clique, which will be our expected $D$.
            One can check directly that they satisfy the requirements.
        \item Suppose that $s_{k+1}+1<t_{k+1}$. 
            We can make a similar argument.
            \qedhere
    \end{enumerate}
\end{proof}

Note that each equivalence class $\calE$ is ordered with respect to the total order $\prec$ in \Cref{Order}. We will divide it into subclasses, for the convenience of our later argument.

\begin{Notation}
    Let $(\calE,L)$ be a marked equivalence class with $\sgn(L)=\bdtau$.  Then, the \emph{subclasses $\calE_{j_1,\ldots,j_k}\subset \calE$ of level $k$} are defined such that the following conditions are satisfied.
    \begin{enumerate}[a]
        \item If $k<n-1$, then $\calE_{j_1,\dots,j_k}=\bigsqcup_{t=1}^{m} \calE_{j_1,\dots,j_k,t}$ is the disjoint union of nonempty subclasses of level $k+1$, for some suitable positive integer $m$.
        \item For any $A,B\in \calE_{j_1,\ldots,j_k}$ with
            $\sgn_{\bdtau}(A)=(p_1,\ldots,p_{n-1})$ and
            $\sgn_{\bdtau}(B)=(q_1,\ldots, q_{n-1})$, one has $p_i=q_i$ for
            all $i=1,\ldots,k$.
        \item \label{notation:subclasses_c}
            It follows from the previous two conditions that every subclass of level $n-1$ contains exactly one maximal clique. Now, if $A\in \calE_{j_1,\ldots,j_{n-1}}$ and $B\in \calE_{\ell_1,\ldots, \ell_{n-1}}$ such that $A\prec B$ (or equivalently, the first nonzero entry of $\sgn_{\bdtau}(A)-\sgn_{\bdtau}(B)$ is negative), then the first nonzero entry of $(j_1, \ldots, j_{n-1})-(\ell_1,\ldots,\ell_{n-1})$ must be negative.
    \end{enumerate}
\end{Notation}

    \begin{Example}
        Let us return to the equivalence class $\calE$ containing the maximal clique
        \[
            A=((1, 1, 1, 3, 2), (0, 2, 1, 3, 2), (0, 1, 2, 3, 2), (0, 1, 2, 2, 3),
            (0, 1, 1, 3, 3))
        \]
        in \Cref{exam:one_max_clique}. The equivalence class $\calE$ contains $12$ different maximal cliques, which we list in \Cref{tabb2} in order. The subclasses of level $4$ containing them, together with their relative signatures, are also provided. In this case, we have $\calE_{1,1}=\{A_1,A_2\}$ and $\calE_{3}=\{A_{10},A_{11},A_{12}\}$. They are subclasses of level $2$ and $1$ respectively.
        \begin{table}[htbp]
            \caption{Subclasses of an equivalence class}
            \label{tabb2}
            \begin{tabular}{cccc}
                \hline
                $i$ & $A_i$ & $\mathcal{E}_{*,*,*,*}$ & $\sgn_{\bdtau}(A_i)$ \\
                \hline
                1 & \((1,1,1,3,2), (1,0,2,3,2), (1,0,2,2,3), (1,0,1,3,3), (0,1,1,3,3)\) & $\mathcal{E}_{1,1,1,1}$ & $(1,2,3,4)$ \\

                2 & \((1,1,1,3,2), (1,0,2,3,2), (1,0,2,2,3), (0,1,2,2,3), (0,1,1,3,3)\) & $\mathcal{E}_{1,1,2,1}$ & $(1,2,4,3)$ \\

                3 & \((1,1,1,3,2), (1,0,2,3,2), (0,1,2,3,2), (0,1,2,2,3), (0,1,1,3,3)\) & $\mathcal{E}_{1,2,1,1}$ & $(1,4,2,3)$ \\

                4 & \((1,1,1,3,2), (1,1,1,2,3), (1,0,2,2,3), (1,0,1,3,3), (0,1,1,3,3)\) & $\mathcal{E}_{2,1,1,1}$ & $(2,1,3,4)$ \\

                5 & \((1,1,1,3,2), (1,1,1,2,3), (1,0,2,2,3), (0,1,2,2,3), (0,1,1,3,3)\) & $\mathcal{E}_{2,1,2,1}$ & $(2,1,4,3)$ \\

                6 & \((1,1,1,3,2), (1,1,1,2,3), (1,1,0,3,3), (1,0,1,3,3), (0,1,1,3,3)\) & $\mathcal{E}_{2,2,1,1}$ & $(2,3,1,4)$ \\

                7 & \((1,1,1,3,2), (1,1,1,2,3), (1,1,0,3,3), (0,2,0,3,3), (0,1,1,3,3)\) & $\mathcal{E}_{2,2,2,1}$ & $(2,3,4,1)$ \\

                8 & \((1,1,1,3,2), (1,1,1,2,3), (0,2,1,2,3), (0,1,2,2,3), (0,1,1,3,3)\) & $\mathcal{E}_{2,3,1,1}$ & $(2,4,1,3)$ \\

                9 & \((1,1,1,3,2), (1,1,1,2,3), (0,2,1,2,3), (0,2,0,3,3), (0,1,1,3,3)\) & $\mathcal{E}_{2,3,2,1}$ & $(2,4,3,1)$ \\

                10 & \((1,1,1,3,2), (0,2,1,3,2), (0,1,2,3,2), (0,1,2,2,3), (0,1,1,3,3)\) & $\mathcal{E}_{3,1,1,1}$ & $(4,1,2,3)$ \\

                11 & \((1,1,1,3,2), (0,2,1,3,2), (0,2,1,2,3), (0,1,2,2,3), (0,1,1,3,3)\) & $\mathcal{E}_{3,2,1,1}$ & $(4,2,1,3)$ \\

                12 & \((1,1,1,3,2), (0,2,1,3,2), (0,2,1,2,3), (0,2,0,3,3), (0,1,1,3,3)\) & $\mathcal{E}_{3,2,2,1}$ & $(4,2,3,1)$ \\
                \hline
            \end{tabular}
        \end{table}
    \end{Example}

\begin{Lemma}
    \label{rmk:keep_order}
    Let $(\calE,L)$ be a marked equivalence class with $\sgn(L)=\bdtau$ and consider a subclass $\calE_{j_1,\dots,j_k}$ of level $k$ with $j_k>1$.  Suppose that $\{s_1,\dots,s_{k-1},s_{k,1},\dots,s_{k,j_k}\}$ is a subset of $[n-1]$ and the signatures of the maximal cliques in $\calE_{j_1, \dots, j_{k-1},\ell}$ have the form $(s_1, \dots, s_{k-1}, s_{k,\ell}, p_{k+1},\dots,p_{n-1})$ for each $\ell\le j_k$, where $p_i\in [n-1]\setminus \{s_1, \dots, s_{k-1}, s_{k,\ell}\}$ for $i=k+1,\dots, n-1$.  Besides, assume that $t$ belongs to $[n-1]\setminus\{s_1, \dots, s_{k-1}, s_{k,1}, \dots, s_{k,j_k-1}\}$ such that $t$ precedes $s_{k,j_k}$ with respect to $\bdtau$. Then, there is no maximal clique $A$ in $\calE$ such that $\sgn(A)$ has the form $(s_1, \dots, s_{k-1}, s_{k,j_k}, t, q_{k+2}, \dots, q_{n-1})$, where $q_i\in [n-1]\setminus \{s_1, \dots, s_{k-1}, s_{k,j_k}, t\}$ for $i=k+2,\dots,n-1$.
\end{Lemma}

\begin{proof}
    Suppose for contradiction that there is a maximal clique $A$ in $\calE$ such that $\sgn(A)$ has the form $(s_1, \dots, s_{k-1}, s_{k,j_k}, t, q_{k+2}, \dots, q_{n-1})$ for suitable $q_i$'s. Since $\bdtau$ is a legitimate signature and $t$ precedes $s_{k,j_k}$ with respect to $\bdtau$, we don't have $s_{k,j_k} \,\triangleleft\, t$ with respect to the partial order in \Cref{def:partial_order}. Consequently, $(s_1, \dots, s_{k-1}, t, s_{k,j_k}, p_{k+2}, \dots, p_{n-1})$ is a legitimate signature within $\calE$ by \Cref{rmk:legal_signature}. Let $\calE_{j_1, \dots, j_{k-1},\ell}$ be the subclass, in which the signatures of the maximal cliques have the form $(s_1, \dots, s_{k-1}, t, r_{k+1},\dots,r_{n-1})$ for suitable $r_i$'s. Since $t$ precedes $s_{k,j_k}$ with respect to $\bdtau$, we have $\ell<j_k$ by the condition \ref{notation:subclasses_c} in our construction of subclasses. But this contradicts the choice of $t$. 
\end{proof}

The following proposition guarantees the linear quotients within an equivalence class.

\begin{Proposition}
    \label{lem:lq_within_group}
    Suppose that $A$ and $B$ are two maximal cliques in an equivalence class
    $\calE$ such that $B\prec A$. Then, there exists a maximal clique $D\in
    \calE$ such that $D\prec A$, and the set difference $A\setminus D$ is a
    singleton set with $A\setminus D\subseteq A\setminus B$.
\end{Proposition}
\begin{proof}
    Suppose that in the marked equivalence class $(\calE,L)$, one has $\sgn(L)=\bdtau$. In addition, assume that $A=(\bda^1,\dots,\bda^n)$ and $B=(\bdb^1,\dots,\bdb^n)\in\calE$. Recall that we previously defined $\diff(A,B)\coloneqq \Set{i:\bda^i\ne \bdb^{i}}$.

    We may further assume that $A\in\calE_{j_1,\dots,j_{k-1},j_k}$ and $B\in\calE_{j_1,\dots,j_{k-1},j_k'}$ with $j_k'<j_k$. Then $k+1=\min \diff(A,B)$ and it is clear that $k\le n-2$.
    Let us consider a set of maximal cliques
    \[
        \calD\coloneqq \Set{F\in \calE: \diff(A,F)\subseteq
        \diff(A,B)\setminus\{k+1\}}.
    \]
    Since $A\in \calD$, the set $\calD$ is not empty. Let $A'$ be the first maximal clique in $\calD$ with respect to $\prec$ and suppose that $\sgn(A')=(s_1, \dots, s_{n-1})$. Obviously, $A'\in \calE_{j_1, \dots, j_k}$ and $A'\preceq A$. Furthermore, we have both $\diff(A', B)\subseteq \diff(A, B)$ and $\diff(A, A')\subseteq \diff(A, B)$. Let $C=(\bdc^{1}, \dots, \bdc^n)$ be the tuple of elements in $\ZZ^n$, such that $\bdc^1=\bda^1$ and $\sgn(C)=(s_1, \dots, s_{k-1}, s_{k+1}, s_k, s_{k+2}, \dots, s_{n-1})$. By \Cref{one_different}, $\diff(C, A')=\{k+1\}$. In addition, we \textbf{claim} that $C$ is a legitimate maximal clique in $\calE$ and $C\prec A'$. With this, if $A=A'$, then we take $D=C$ since $\diff(A, B)\supseteq \diff(A, C)=\{k+1\}$. If instead $A\ne A'$, then we will replace $B$ by $A'$. Our proof will be done by induction on the cardinality $\#\diff(A, B) $, since $\diff(A, A') \subsetneq \diff(A, B)$ and $A'\prec A$.

    It remains to prove the above claim about $C$. Since $A\in\calE_{j_1, \dots, j_{k-1}, j_k}$, we may assume that the signatures of the maximal cliques in $\calE_{j_1, \dots, j_{k-1}, \ell}$ have the form $(s_1, \dots, s_{k-1}, s_{k,\ell}, p_{k+1}, \dots, p_{n-1})$ for each $\ell< j_k$ with suitable $p_i$'s.  From \Cref{rmk:keep_order} we derive that either $s_k$ precedes $s_{k+1}$ with respect to $\bdtau$, or $s_{k+1}\in \{s_{k,1},\dots,s_{k,j_k-1}\}$.

    \begin{enumerate}[a]
        \item Suppose that $s_k$ precedes $s_{k+1}$ with respect to $\bdtau$.  Since $B\in\calE_{j_1, \dots, j_{k-1}, j_k'}$ with $j_k'<j_k$, we can write $\sgn(B)=(s_1, \dots, s_{k-1}, s_{k, j_k'}, p_{k+1}, \dots, p_{n-1})$ with suitable $p_i$'s. Since $j_k'<j_k$, $s_{k, j_k'}$ precedes $s_k$ with respect to $\bdtau$. As a result of our assumption here, $s_{k, j_k'}$ also precedes $s_{k+1}$ with respect to $\bdtau$.  Without loss of generality, we may assume that $\diff(A',B)=\{k+1,k+2,\dots,r\}$ is a ``continuous'' segment.  Applying \Cref{lem:exchange} to $A'$ and $B$, we can find some maximal clique $A''$ in $\calE$ such that $\diff(A', A'')\subseteq \diff(A', B)$ and $\sgn(A'')=(s_1, \dots, s_k, s_{k, j_k'}, r_{k+2}, \dots, r_{n-1})$ for suitable $r_i$'s.  In particular, $k+1\notin \diff(A, A'')$. Meanwhile, we observe that
            \[
                \diff(A, A'')\subseteq \diff(A, A') \cup \diff(A', A'')
                \subseteq \diff(A,B)\cup \diff(A',B) =
                \diff(A,B).
            \]
            Thus, the maximal clique $A''$ belongs to the previously defined set $\calD$. But since $s_{k,j_k'}$ precedes $s_{k+1}$ with respect to $\bdtau$ while $\sgn(A')=(s_1,\dots,s_{n-1})$, this contradicts our choice of $A'$.

        \item Suppose instead that $s_{k+1}=s_{k, \ell}$ for some $\ell<j_k$. Then, $\bdc^{k+1}$ is a legitimate tuple by the existence of $\calE_{j_1, \dots, j_{k-1}, \ell}$. Meanwhile, notice that $A'$ is a legitimate maximal clique such that $\diff(C, A')= \{k+1\}$. Consequently, $C$ is also a legitimate maximal clique. Since $\ell<j_k$, the index $s_{k+1}$ precedes $s_k$ with respect to $\bdtau$. Thus, $C\prec A'$, fully confirming our claim about $C$.
            \qedhere
    \end{enumerate}
\end{proof}

Next, we consider the linear quotients across equivalence classes.

\begin{Lemma}
    \label{rmk:rL_exists}
    Suppose that $\calE$ is an equivalence class such that $\rank(\calE)>0$.  Let $\kappa_1$, $\kappa_2$, and $L$ be as introduced in \Cref{def:L}. Then, we have $\kappa_1+1<\kappa_2-1$.
\end{Lemma}
\begin{proof}
    Suppose for contradiction that $\kappa_1+1\ge \kappa_2-1$. Since $\kappa_1<\kappa_2$, we obtain either $\kappa_1+1=\kappa_2$ or $\kappa_1+1=\kappa_2-1$. Since $n\ge 3$, $\kappa_1\le n-1$, and $\kappa_2\ge 2$, we encounter the following four cases. 
    \begin{enumerate}[i]
        \item Suppose that $\kappa_1=n-1$. Then $\bda^1=(1, 0, \dots, 0, a_n^1)$ with $0\le a_n^1\le \alpha_n-1$. Consequently, $\bda^n=(0, \dots, 0, a_n^1+1)$, which has to be $\bdeta$. 
        \item Suppose that $\kappa_2=2$. Then $\bda^1=(a_1^1,\alpha_2,\dots,\alpha_{n-1},\alpha_n-1)$ with $a_1^1>1$. Consequently, $\bda^n=(a_1^1-1,\alpha_2,\dots,\alpha_n)$, which has to be $\bdeta$.
        \item Suppose that $\kappa_1+1=\kappa_2$ and $2\le \kappa_1\le n-2$. Then $\bda^1=(1, 0, \dots, 0, \alpha_{\kappa_1+1}, \dots, \alpha_{n-1}, \alpha_n-1)$.  Consequently, $\bda^n= (0, \dots, 0, \alpha_{\kappa_1+1}, \dots, \alpha_n)$, which has to be $\bdeta$. 
        \item Suppose that $\kappa_1+1=\kappa_2-1$ and $1\le \kappa_1\le n-2$. Then 
            \[
                \bda^1=(1, 0, \dots, 0, a_{\kappa_1+1}^1, \alpha_{\kappa_1+2}, \dots, \alpha_{n-1}, \alpha_n-1)
            \]
            with $0<a_{\kappa_1+1}^1<\alpha_{\kappa_1+1}$. Consequently,
            $\bda^n=(0, \dots, 0, a_{\kappa_1+1}^1, \alpha_{\kappa_1+2}, \dots,
            \alpha_{n})$, which has to be $\bdeta$. 
    \end{enumerate}
    In each case, we always have $\rank(\calE)=0$, which is a contradiction. 
\end{proof}

\begin{Proposition}
    \label{lem:Aj+1}
    Let $A=(\bda^1,\dots,\bda^n)$ be a maximal clique in the equivalence class
    $\calE$ such that $rA$ does not exist. Suppose that $B\prec A$ is a
    maximal clique such that $B\notin \calE$. Then, there exists another
    maximal clique $D\in \calE$ and some $j\in [n-1]$ such that $D\prec A$
    and $A\setminus D=\{\bda^{j+1}\}\subseteq A\setminus B$.
\end{Proposition}

\begin{proof}
    We will prove this by using the notation and arguments of \Cref{def:L}.
    As $\rank(\calE)>0$ by the existence of $B$, we showed in \Cref{rmk:rL_exists} that $\kappa_1+1<\kappa_2-1$, and the first maximal clique $L$ in $\calE$ satisfies
    \[
        \sgn(L)=(\tau_1, \dots, \tau_{\kappa_2-\kappa_1-2}, \underbrace{n-1, n-2,
        \dots, \kappa_2-1}, \underbrace{1, 2, \dots, \kappa_1}),
    \]
    for appropriate $\tau_i$'s.

    Suppose that $\sgn(A)=(s_1, \dots, s_{n-1})$ and let $W$ be the set
    \[
        [n-1] \setminus (\{n-1, n-2, \dots, \kappa_2-1\}\sqcup
        \{1, 2, \dots, \kappa_1\})= \{\kappa_1+1, \dots, \kappa_2-2\}.
    \]
    Since $\kappa_1+1<\kappa_2-1$, $W$ is not empty. Meanwhile, since $rA$ does not exist, $s_1=1$ or $n-1$ by \Cref{lem:rA_not_exists}, and $s_1\notin W$.
    In the following, let $j\ge 1$ be the smallest such that $s_{j+1}\in W$. We can construct a maximal clique $D$ in $\calE$ such that $\sgn(D)=(s_1, \dots, s_{j-1}, s_{j+1}, s_j, s_{j+2}, \dots, s_{n-1})$.

    To confirm the legitimacy of $D$, notice that $s_j\notin W$ while $s_{j+1}\in W$. 
    Since $s_{j+1}$ precedes $s_j$ in $\sgn(L)$ while $s_j$ precedes $s_{j+1}$ in $\sgn(A)$, these two indices
    are not comparable in the poset of obstructions, i.e., they can exchange positions in any legitimate signature. As $\sgn(A)$ is a legitimate signature, so is $\sgn(D)$, namely, $D$ is a maximal clique in $\calE$.

    Next, since $s_{j+1}$ precedes $s_j$ in $\sgn(L)$, we deduce that $D\prec
    A$. From \Cref{one_different} we also have $A\setminus D=\{\bda^{j+1}\}$.
    It remains to show that this $\bda^{j+1}\notin B$. Suppose for
    contradiction that this is not true. Then, we can write $B=(\bdb^1, \dots,
    \bdb^n)$ and $\bdb^{j'}=\bda^{j+1}$ for some $j'$. Since $B\prec A$, we
    must have $j'\le j+1$, by rank reason. At the same time, given
    \Cref{lem:po}, we can find $i_1\le \kappa_1$ and $i_2\le n-\kappa_2$ such
    that $\{s_1, \dots, s_j\}=\{1, \dots, i_1\}\sqcup \{n-i_2, \dots, n-1\}$,
    by the choice of $j$. Whence, we can write $\bda^{j+1}=(\underbrace{0,
    \dots, 0}_{i_1}, a^{j+1}_{i_1+1}, \dots, a^{j+1}_{n-i_2},
    \underbrace{\alpha_{n-i_2+1}, \dots, \alpha_n}_{i_2})$. Since $\bdb^{j'}=
    \bda^{j+1}$, we obtain that $1, 2, \dots, i_1, n-i_2, \dots, n-1\notin
    \Delta(\bdb^{j'}, \bdb^n)$.  Thus, if $\sgn(B)=(q_1, \dots, q_{n-1})$, then
    $\{1, \dots, i_1, n-i_2, \dots, n-1\}\subseteq \{q_1, \dots, q_{j'-1}\}$,
    forcing $j\le j'-1$. Therefore, $j'=j+1$.  It is now clear that $\{q_1,
    \dots, q_j\}=\{1, \dots, i_1\}\sqcup \{n-i_2, \dots, n-1\}$, and
    $\Delta(\bdb^1, \bdb^{j+1})= \Delta(\bda^1, \bda^{j+1})$. As
    $\bdb^{j+1}=\bda^{j+1}$, it follows that $\bdb^1=\bda^1$ and $B\in \calE$
    as well. This contradicts the assumption that $B\notin \calE$.
\end{proof}

Finally, we are ready to show the announced linear quotient result.

\begin{Theorem}
    \label{prop:linearQuotient}
    The total order $\prec$ of \Cref{Order} induces a linear quotient order of the monomial generating set $G(\ini(J)^{\vee})$.
\end{Theorem}

\begin{proof}
    Take arbitrary maximal cliques $A$ and $B$ such that $B\prec A$. Given \Cref{translate_lq_sets}, it suffices to find suitable maximal clique $D$ with $D\prec A$ such that $A\setminus D$ is a singleton set with $A\setminus D \subseteq A\setminus B$. Suppose that $A=(\bda^1, \dots, \bda^n)$ and $A\in\calE$. We have two cases.
    \begin{enumerate}[a]
        \item Assume that $B\in \calE$. We apply \Cref{lem:lq_within_group}
            for the existence of such $D$.
        \item Assume that $B\notin \calE$.
            \begin{enumerate}[i]
                \item If $rA$ exists, then $A\setminus rA=\{\bda^1\}$. Note that $\bda^1\notin B$ for rank reasons. So in this case we take $D=rA$.
                \item Assume that $rA$ does not exist. We apply \Cref{lem:Aj+1} for the existence of such $D$. \qedhere
            \end{enumerate}
    \end{enumerate}
\end{proof}

\section{Applications}
\label{Sec:app}

This section is devoted to two applications of the linear quotient structure that we established earlier.  Here, we continue to assume that $\calA_{d,\bdalpha}$ is the Veronese type algebra of \Cref{set:ideal_veronese} and $\bdeta$ is the tuple given in \Cref{eta}. Meanwhile, $J$ is the presentation ideal so that $\calA_{d,\bdalpha}=\KK[\bdT]/J$.

\subsection{Regularity of the algebra}
First of all, we determine the Castelnuovo--Mumford regularity of the algebra $\calA_{d,\bdalpha}$.

For each maximal clique $A$ in the graph $\calG=\calG(d,\bdalpha)$, we will denote the minimal number of generators of the linear quotient ideal $\braket{\bdT_{B^\complement} : B\prec A}:_{\KK[\bdT]} \bdT_{A^\complement}$ by $\omega_{\bdalpha}(A)$, or $\omega(A)$ for short. 
By the proof of \Cref{prop:linearQuotient}, we have
\begin{equation}
    \omega(A)= \# \Set{B\prec A: \diff(A,B)\text{ is a singleton set}}.
    \label{eqn:omeage_A}
\end{equation}
Furthermore, by \Cref{MaxLenghtToPd}, we have the following formula:
\begin{equation}
    \label{pdIsMax}
    \pd((\ini(J))^\vee)=\max\limits_{A} \omega(A).
\end{equation}
Since $\pd((\ini(J))^\vee) =\reg (\calA_{d,\bdalpha})$ by \Cref{regToPd}, the task is now clear: 
find the largest $\omega(A)$.
We start with a quick estimate.

\begin{Lemma}
    \label{lem:omega_upper_bound}
    We have $\omega(A)\le n-1$ for each maximal clique $A=(\bda^1,\dots,\bda^n)$. 
\end{Lemma}
\begin{proof}
    Since $\braket{\bdT_{B^\complement} : B\prec A}:_{\KK[\bdT]} \bdT_{A^\complement}$ is linear, it follows from \Cref{translate_lq_sets} that 
    \[
        \omega(A)=\#\Set{\bda^i:\text{there exists some $B\prec A$ with
        $\{\bda^i\}=A\setminus B$}}.
    \]
    Therefore, it suffices to show that there is no $B$ with $B\prec A$ such that $\{\bda^n\}=A\setminus B$. Suppose for contradiction that this is not true. By \Cref{cor:Delta_maximal_clique}, we have either $B=(\bdb,\bda^1,\bda^2,\dots,\bda^{n-1})$ or $B=(\bda^1,\bda^2,\dots,\bda^{n-1},\bdb)$ for some suitable tuple $\bdb$. In the first case, $\rank(B)=\rank(A)+1$, which contradicts the assumption that $B\prec A$. In the second case, we have $B=A$ by \Cref{cor:Delta_maximal_clique}, which is also a contradiction.
\end{proof}

It is natural to ask: when do we have $\omega(A)=n-1$?  Let us start with a simple observation. In any equivalence class $\calE$, every maximal clique $A$ uniquely corresponds to the signature $\sgn(A)$, which is a permutation in $\frakS_{n-1}$. Thus, it is clear that $\calE$ contains at most $(n-1)!$ elements. In the following, we classify when an equivalent class $\calE$ contains exactly $(n-1)!$ elements.

\begin{Lemma}
    \label{prop:full_permutation}
    Let $\calE$ be an equivalence class such that every maximal clique in it begins with $\bda^1=(a_1^1,\ldots,a_n^1)$. Then, the following are equivalent:
    \begin{enumerate}[a]
        \item the cardinality $\#\calE=(n-1)!$;
        \item the poset of obstructions in \Cref{def:partial_order} is trivial for $\calE$;
        \item one has $1 \le a_1^1\le \alpha_1$, $1\le a_j^1\le \alpha_j-1$ for all $2\le j\le n-1$, and $0\le a_n^1\le \alpha_n-1$.
    \end{enumerate}
\end{Lemma}
\begin{proof}
    The cardinality $\# \calE=(n-1)!$ if and only if every permutation in $\frakS_{n-1}$ is 
    the signature of some maximal clique in $\calE$. But the latter is equivalent to saying that the poset of obstructions in \Cref{def:partial_order} is trivial for $\calE$, namely, $1\le a_j^1\le \alpha_j-1$ for all $2\le j\le n-1$. That $1 \le a_1^1\le \alpha_1$ and $0\le a_n^1\le \alpha_n-1$ is automatic for all such $\bda^1$ in view of \Cref{cor:Delta_maximal_clique}.
\end{proof}

In the following, we characterize when $\pd ((\ini(J))^\vee)$ is exactly $n-1$.

\begin{Proposition}
    \label{prop:reg_middle_part}
    The following conditions are equivalent:
    \begin{enumerate}[a]
        \item \label{prop:reg_middle_part_a}
            the projective dimension $\pd ((\ini(J))^\vee)=n-1$;
        \item \label{prop:reg_middle_part_b}
            there exists a maximal clique $A$ such that $\omega(A)=n-1$;
        \item \label{prop:reg_middle_part_c}
            one has $n\le d \le \sum_{i=1}^n(\alpha_i-1)$;
        \item \label{prop:reg_middle_part_d}
            there exists a maximal clique $A=(\bda^1,\ldots,\bda^n)$ with
            $\bda^1=(a_1^1,\ldots,a_n^1)$ such that
            \[
                2\le a_1^1 \le \alpha_1, \quad 1\le a_j^1\le \alpha_j-1\text{ for all } 2\le j\le n-1, \quad \text{and}\quad 0\le a_n^1\le
                \alpha_n-2.
            \]
    \end{enumerate}
\end{Proposition}
\begin{proof}
    The equivalence of \ref{prop:reg_middle_part_a} and \ref{prop:reg_middle_part_b} is clear from the explanation before \Cref{prop:full_permutation}.

    Next, we show the equivalence of \ref{prop:reg_middle_part_c} and \ref{prop:reg_middle_part_d}.  Since $|\bda^1|=d$, we can easily deduce \ref{prop:reg_middle_part_c} from \ref{prop:reg_middle_part_d}. Conversely, if \ref{prop:reg_middle_part_c} is satisfied, we can easily find suitable $\bda^1 =(a_1^1,\ldots,a_n^1) \in \NN^n$ as in \ref{prop:reg_middle_part_d}. Let $A=(\bda^1, \dots, \bda^n)$ be the tuple of elements in $\ZZ^n$ such that $\sgn(A)=(1, 2, \dots, n-1)$.  It can be verified directly that $A$ is a legitimate maximal clique. Thus, we get \ref{prop:reg_middle_part_d}.

    In the following, we show the equivalence of \ref{prop:reg_middle_part_b}
    and \ref{prop:reg_middle_part_d}.

    \begin{enumerate}[align=left, leftmargin=*]
        \item[\ref{prop:reg_middle_part_b} $\Leftarrow$
            \ref{prop:reg_middle_part_d}:] Suppose that the condition in
            \ref{prop:reg_middle_part_d} is satisfied.  Then, the equivalence
            class $\calE$ to which $A$ belongs has exactly $(n-1)!$ maximal
            cliques by \Cref{prop:full_permutation}, and every permutation in
            $\frakS_{n-1}$ is legitimate as a signature with respect to
            $\calE$.  Without loss of generality, we may assume that $A$ is the
            very last maximal clique of $\calE$ with $\sgn(A)=(s_1, \dots,
            s_{n-1})$.  For each $k=2, 3, \dots, n-1$, we consider the maximal
            clique $B^k$ in $\calE$ with $\sgn(B^k)=(s_1, \dots, s_{k-2}, s_k,
            s_{k-1}, s_{k+1}, \dots, s_{n-1})$. Then $B^k\prec A$ and
            $A\setminus B^k=\{\bda^{k}\}$ by \Cref{one_different}. Meanwhile,
            since $a_n^1\le \alpha_n-2$, it follows from
            \Cref{cor:Delta_maximal_clique} that $\bda^n$ is not the tuple
            $\bdeta$ of \Cref{eta}.  In other words, $\rank(\calE)>0$.  Thus,
            $rA$ exists and $A\setminus rA=\{\bda^1\}$ by
            \Cref{lem:rA_not_exists}.  In conclusion,
            $\braket{\bdT_{C^\complement}: C\prec A}:_{\KK[\bdT]}
            \bdT_{A^\complement} =\braket{T_{\bda^1}, \dots, T_{\bda^{n-1}}}$
            is linear and has the maximal size by the proof of
            \Cref{lem:omega_upper_bound}. In particular,
            \ref{prop:reg_middle_part_b} holds.

        \item[\ref{prop:reg_middle_part_b} $\Rightarrow$
            \ref{prop:reg_middle_part_d}:] Suppose that the condition in
            \ref{prop:reg_middle_part_b} is satisfied. Then, there exists a
            maximal clique $A=(\bda^1, \dots, \bda^{n})$ in some equivalence
            class $(\calE,L)$ such that the quotient ideal $Q\coloneqq
            \braket{\bdT_{B^\complement}: B\prec A}:_{\KK[\bdT]}
            \bdT_{A^\complement}$ is linear with $n-1$ minimal monomial
            generators.  In view of \Cref{lem:omega_upper_bound} and its proof,
            this implies that $Q=\braket{T_{\bda^1},T_{\bda^2}, \dots,
            T_{\bda^{n-1}}}$.  Since $T_{\bda^1}\in Q$, $rA$ must exist and
            $\rank(\calE)>0$. Additionally, for each $k\in \{2, \dots, n-1\}$,
            we have a maximal clique $B^k$ such that $B^k \prec A$ and
            $A\setminus B^k=\{\bda^{k}\}$. Since $B^k$ obviously starts with
            $\bda^1$, this maximal clique belongs to $\calE$. Now, suppose that
            $\sgn(A)=(s_1, \dots, s_{n-1})$. It follows from
            \Cref{one_different} that $\sgn(B^k)=(s_1, \dots, s_{k-2}, s_k,
            s_{k-1}, s_{k+1}, \ldots, s_{n-1})$.  Since $B^k\prec A$, we must
            have $s_{k} < s_{k-1}$ when considering the lexicographical order
            at the end of \Cref{Order}\ref{lem:lq_cross_groups}. Since this
            holds for any $k\in \{2, \dots, n-1\}$, we conclude that $s_{n-1}<
            s_{n-2} <\cdots < s_1$ with respect to this lexicographical order.
            Whence, $A$ is the last maximal clique of $\calE$ and
            $\sgn(L)=(s_{n-1},\dots,s_1)$ is the reverse of $\sgn(A)$.
            Consequently, the poset of obstructions considered in
            \Cref{def:partial_order} is trivial, and we have $0<a_i^1<\alpha_i$
            for $i\in \{2,\dots,n-1\}$ by \Cref{prop:full_permutation}.

            Furthermore, since $A$ is legitimate, $1\le a_1^1\le \alpha_1$ and
            $0\le a_n^1\le \alpha_n-1$ by \Cref{cor:Delta_maximal_clique}. If
            $a_1^1=1$, then $\sgn(L)$ takes the form $(\tau_1, \dots,
            \tau_{n-2}, 1)$ by the requirement in \Cref{def:L}.  Whence,
            $s_1=1$. But this implies that $rA$ does not exist by
            \Cref{lem:rA_not_exists}, a contradiction.  Similarly, we can prove
            that $a_n^1\le \alpha_n-2$. Thus, \ref{prop:reg_middle_part_d}
            holds. \qedhere
    \end{enumerate}
\end{proof}

In what follows, we consider the case where the equivalent requirements of \Cref{prop:reg_middle_part} are not satisfied. 
Before doing so, we introduce a reduction.   

\begin{Remark}
    Suppose that $I$ is an equigenerated monomial ideal with minimal monomial
    generators $\bdx^{\bda^1}, \dots, \bdx^{\bda^t}$ in $\KK[x_1, \dots, x_n]$.
    In addition, suppose that $\bdb$ is a tuple in $\NN^n$ such that
    $\bdb-\bda^i\in \NN^n$ for each $i$.  Whence, we will call
    $I^{[\bdb]}\coloneqq \braket{\bdx^{\bdb-\bda^1}, \dots,
    \bdx^{\bdb-\bda^t}}$ the \emph{generalized Newton dual of $I$ with respect
to $\bdb$}. It follows from \cite[Theorem 3.1]{MR4221772} that the algebra
$\KK[I]$ is isomorphic to $\KK[I^{[\bdb]}]$.  Of course, we are only interested
in the case where $I=I_{d,\bdalpha}$, and $\bdb=\bdalpha$.  In this case, we
have
\begin{equation}
    \calA_{d,\bdalpha}= \KK[I] \isom \KK[I^{[\bdalpha]}]= 
    \calA_{|\bdalpha|-d,\bdalpha}.
    \label{eqn:reduction_lambda}
\end{equation}
\end{Remark}

\begin{Proposition}
    \label{prop:reg_two_extremal_case}
    Suppose that either $n>d$ or $\sum_{i=1}^n(\alpha_i-1)<d$. Let
    $d'=\min(d,|\bdalpha|-d)$. Then, $\pd
    ((\ini(J))^\vee)=\floor{n-\frac{n}{d'}}$.
\end{Proposition}
\begin{proof}
    It follows from the isomorphism in Equation \eqref{eqn:reduction_lambda}
    that we can assume that $d'=d$ and $n>d$. From this, we also deduce that
    $2d\le |\bdalpha|$.

    As the first step, we show that $\pd ((\ini(J))^\vee)\ge
    \floor{n-\frac{n}{d}}$.  In the extremal case when $d=1$, since
    $|\bdeta|=d$, one obviously has $\bdeta=(0,\dots,0,1)$ and $I_{d,\bdalpha}$
    is the graded maximal ideal of $S$. Whence, the claimed formula is
    guaranteed by \cite[Theorem 4.2]{MR2955237}. Therefore, in the following,
    we will assume that $d\ge 2$.  By \Cref{pdIsMax}, we need to find a maximal
    clique $A$ such that $ \omega(A)\geq \floor{n-\frac{n}{d}}$.  Suppose that
    $n-1=pd+q$ such that $p =\floor{\frac{n-1}{d}}$ and $0\le q< d$. Then,
    $\floor{n-\frac{n}{d}} =n-1-p$.  Let $\calE$ be the equivalence class of
    maximal cliques, all of which start with the vertex
    \[
        \bda^1\coloneqq (
        \underbrace{\underbrace{1, 0, \dots, 0}_{p}, \dots,
        \underbrace{1, 0, \dots, 0}_{p}}_{d-q}, 
        \underbrace{\underbrace{1, 0, \dots, 0}_{p+1}, \dots, 
        \underbrace{1, 0, \dots, 0}_{p+1}}_{q}, 0)
    \]
    in $\calG(d,\bdalpha)$. Thus, the maximal cliques in $\calE$ all end with
    $\bda^n$ such that $\Delta(\bda^1,\bda^n)=[1,n)$ by
    \Cref{cor:Delta_maximal_clique}. It is clear that $\bda^n= (0, a^1_2,
    \ldots, a^1_{n-1},1)$. If $\alpha_n\ge 2$, since $\bdeta=(*,\dots,*,\alpha_n)$, we have $\bda^n\ne \bdeta$.
    Therefore, $\rank(\calE)>0$. If instead $\alpha_n=1$, then $\bdalpha=(1,\dots,1)$. Whence, $2d\le |\bdalpha|=n$. 
    We still have $\rank(\calE)>0$. Otherwise, we will have
    \[
        \bda^n= (\underbrace{0, \dots, 0}_{p}, 
        \underbrace{\underbrace{1, 0, \dots, 0}_{p}, \dots,
        \underbrace{1, 0, \dots, 0}_{p}}_{d-q-1}, 
        \underbrace{\underbrace{1, 0, \dots, 0}_{p+1}, \dots, 
        \underbrace{1, 0, \dots, 0}_{p+1}}_{q}, 1)=\bdeta.
    \]
    Since $p\ge 1$, by the description of $\bdeta$ in \Cref{eta},
    the only possibility is $q=0$ and $p=1$. Whence, $n=d+1$. But as $n\ge 3$ and $2d\le n$, this is impossible.

    To simplify the following proof, we write accordingly 
    \[
        (1,2,\dots,n-1)=(\underbrace{\underbrace{s_{1}^1,\dots,s_{p}^1}_{p},
        \dots,\underbrace{s_{1}^{d-q},\dots,s_{p}^{d-q}}_{p}}_{d-q},
        \underbrace{\underbrace{s_{0}^{d-q+1},\dots,s_{p}^{d-q+1}}_{p+1},
        \dots,\underbrace{s_{0}^{d},\dots,s_{p}^{d}}_{p+1}}_{q}).
    \]
    Then, we have
    \begin{equation}
        \label{eqn:relation}
        s_i^{\ell} \, \triangleleft\, s_{j}^{\ell} \qquad \text{for all $i<j$
        and all $\ell$}
    \end{equation}
    in the poset of obstructions defined in \Cref{def:partial_order}. Furthermore, we have 
    \begin{align}
        &s_1^{i+1} \, \triangleleft\, s_{p}^i \qquad
        \text{
        if $1 \le i \le d-q-1$ and $\alpha_{s_1^{i+1}}=1$,}
        \label{eqn:relation_2}
        \intertext{and}
        & s_0^{j+1} \, \triangleleft\, s_{p}^{j} \qquad 
        \text{
        if $d-q\le j \le d-1$ and $\alpha_{s_0^{j+1}}=1$,}
        \label{eqn:relation_3}
    \end{align}
    in the poset of obstructions.
    Indeed, they are the generating relations of that poset.  Since the
    Castelnuovo--Mumford regularity of $\calA_{d,\bdalpha}$ is independent of
    the rules we impose in Section \ref{Sec:Shelling}, we can further require
    that
    \begin{align*}
        \sgn(L)\coloneqq
        (\underbrace{s_{0}^{d},\dots,s_{p-1}^{d}}_{p},
        &\overbrace{\underbrace{s_{0}^{d-1},\dots,s_{p}^{d-1}}_{p+1},
        \dots,\underbrace{s_{0}^{d-q+1},\dots,s_{p}^{d-q+1}}_{p+1}}^{q-1},
        \\[0.5em]
        & \qquad \qquad \underbrace{\underbrace{s_{1}^{d-q},
                \dots,s_{p}^{d-q}}_{p},\dots,\underbrace{s_{1}^2,
        \dots,s_{p}^2}_{p}}_{d-q-1},s_{p}^d,
        \underbrace{s_{1}^1,\dots,s_{p}^1}_{p})
    \end{align*}
    when $q \ge 1$. If instead $q=0$, we then require that
    \[
        \sgn(L) \coloneqq (\underbrace{s_1^d, \dots, s_{p-1}^d}_{p-1},
        \underbrace{\underbrace{s_{1}^{d-1}, \dots, s_{p}^{d-1}}_{p}, \dots,
        \underbrace{s_{1}^2, \dots, s_{p}^2}_{p}}_{d-2}, s_{p}^d,
        \underbrace{s_{1}^1, \dots, s_{p}^1}_{p}).
    \]
    In any case, the prescribed $\sgn(L)$ is a legitimate signature by \Cref{rmk:legal_signature}, since it satisfies the requirements of Equations \eqref{eqn:relation}, \eqref{eqn:relation_2}, \eqref{eqn:relation_3}, and \Cref{def:L}.
    Now, let $A$ be the maximal clique starting from $\bda^1$ such that
    \[
        \sgn(A)\coloneqq (
        \underbrace{s_0^{d-q+1}, s_0^{d-q+2}, \dots, \fbox{$s_0^{d}$}}_{q}, 
        \underbrace{\underbrace{s_1^1, s_1^2, \dots, \fbox{$s_1^d$}}_d, \dots,
        \underbrace{s_{p-1}^1, s_{p-1}^2, \dots, \fbox{$s_{p-1}^d$}}_d}_{p-1},
        \underbrace{s_{p}^1, s_{p}^d, s_{p}^{2}, \dots, \fbox{$s_{p}^{d-1}$}}_d).
    \]
    Then $A$ is legitimate by \Cref{rmk:legal_signature}, since it satisfies the requirements of Equations \eqref{eqn:relation}, \eqref{eqn:relation_2}, and \eqref{eqn:relation_3}. Note that all end positions of the underbraced segments in $\sgn(A)$ are boxed. Suppose that we also write 
    $\sgn(A)=(t_1, \dots, t_{n-1})$.
    For each $k\in [n-2]$ such that $t_k$ is not boxed,
    we can find a tuple $B^k=(\bdb_1^k,\dots,\bdb_n^k)$ of elements in $\ZZ^n$ such that $\bdb_1^k=\bda^1$ and
    \[
        \sgn(B^k)=(t_1,\dots,t_{k-1},t_{k+1},t_{k},t_{k+2},\dots,t_{n-1}).
    \]
    Note that $B^k$ is also a legitimate maximal clique in $\calE$, since it satisfies the requirements of Equations \eqref{eqn:relation}, \eqref{eqn:relation_2}, and \eqref{eqn:relation_3}. Moreover, $B^k\prec A$ due to our choice of $\sgn(L)$, and $\diff(A,B^k)$ is a singleton set by \Cref{one_different}. Consequently, we have a natural lower bound for the minimal number of generators:
    \[
        \mu(\braket{\bdT_{B^\complement}:B\in\calE\text{ and }B\prec
        A}:_{\KK[\bdT]} \bdT_{A^\complement}) \ge
        \begin{cases}
            n-2-p, & \text{if $q\ne 0$,} \\
            n-1-p, & \text{if $q=0$}
        \end{cases}
    \]
    by the first part of the proof of \Cref{prop:linearQuotient}. Notice that $q=0$ if and only if $\sgn(A)$ starts with $1$, if and only if $rA$ does not exist by \Cref{lem:rA_not_exists}.  Thus,
    \[
        \omega(A)=\mu(\braket{\bdT_{B^\complement}:B\prec A}:_{\KK[\bdT]}
        \bdT_{A^\complement}) \ge n-1-{p}=\floor{n-\frac{n}{d}}
    \]
    by the second part of the proof of \Cref{prop:linearQuotient}.  Therefore, we get $\pd ((\ini(J))^\vee)=\max\limits_{B} \omega(B)\ge \floor{n-\frac{n}{d}}$, as planned.

    As the second step, we show
    that $\pd ((\ini(J))^\vee)\le \floor{n-\frac{n}{d}}$. For this purpose,
    we introduce the new tuple $\bdalpha'=(d,\dots,d)$. Notice that the sets of maximal cliques satisfy that $\MC(\calG(d,\bdalpha))\subseteq \MC(\calG(d,\bdalpha'))$.
    For any fixed $A=(\bda^1,\dots,\bda^n) \in \MC(\calG(d,\bdalpha))$, let $\calE(\bdalpha)$
    (resp.~$\calE(\bdalpha')$) be the equivalence class in
    $\calG(d,\bdalpha)$ (resp.~$\calG(d,\bdalpha')$) to which $A$ belongs. Let $\bdeta$ (resp.~$\bdeta'$) be the tuple given in \Cref{eta} for $\bdalpha$ (resp.~$\bdalpha'$). It
    is clear that the post of obstructions of $\calE(\bdalpha')$ is a
    subposet of $\calE(\bdalpha)$. 
    Let $\kappa_1$ and $\kappa_2$ (resp.~$\kappa_1'$ and $\kappa_2'$) be the index
    defined in \Cref{def:L}
    for $\bdalpha$ (resp.~$\bdalpha'$), then we have
    $\kappa_1=\kappa_1'$ by the definition. 
    As for $\kappa_2$ and $\kappa_2'$, notice first that $a^1_n\le \alpha'_n-1=d-1$ and $|\bda^1|=d$. If
    $a^1_n=d-1$, we must have $\bda^1=(1,0,\ldots,0,d-1)$ and $\bda^n=(0,\ldots,0,d)=\bdeta=\bdeta'$. Whence, $\calE(\bdalpha)=\calE(\bdalpha')$ has rank $0$ and contains exactly one maximal clique:
    \[
        A=((1,0,\ldots,0,d-1),(0,1,0,\ldots,0,d-1),\ldots,(0,\ldots,0,d)).
    \] 
    In particular, $\omega_\bdalpha(A)=0=\omega_{\bdalpha'}(A)$.
    On the other hand, 
    if $a^1_n< d-1$, then $\kappa_2'$ does not exist. 
    As a result, the special $L$ we designate to $\calE(\bdalpha)$ in \Cref{def:L}
    also works
    for the equivalence $\calE(\bdalpha')$. Consequently, we have
    \[
        \Set{B\in \MC(\calG(d,\bdalpha)):B\prec_{\bdalpha} A}=\Set{B\in
        \MC(\calG(d,\bdalpha')):B\prec_{\bdalpha'}A} \cap \MC(\calG(d,\bdalpha)),
    \]
    and $\omega_{\bdalpha}(A)\le \omega_{\bdalpha'}(A)$
    by \Cref{eqn:omeage_A}.
    Then it is easy to deduce that
    \begin{align*}
        \pd ((\ini(J(\bdalpha)))^\vee) &= \max_{A\in \MC(\calG(d,\bdalpha))}
        \omega_{\bdalpha}(A) \le 
        \max_{A\in \MC(\calG(d,\bdalpha))} \omega_{\bdalpha'}(A) 
        \\
        &\le
        \max_{A\in \MC(\calG(d,\bdalpha'))} \omega_{\bdalpha'}(A) =
        \pd ((\ini(J(\bdalpha')))^\vee).
    \end{align*}
    Since $\pd ((\ini(J(\bdalpha')))^\vee)=
    \reg(\calA_{d,\bdalpha'})=\floor{n-\frac{n}{d}}$ by \cite[Theorem
    4.2]{MR2955237}, this completes the proof.
\end{proof}

We can summarize the above results and state the first main theorem of this
section.

\begin{Theorem}
    \label{regSUM}
    Let $\calA_{d,\bdalpha}$ be the Veronese type algebra in
    \Cref{set:ideal_veronese}.
    Set $d'=\min(d,|\bdalpha|-d)$. Then, $\reg (\calA_{d,\bdalpha})=\floor{n-\frac{n}{d'}}$.
\end{Theorem}
\begin{proof}
    Let $J$ be the presentation ideal of $\calA_{d,\bdalpha}$.  From
    \Cref{thm:definingIdeals}, \cite[Corollary 2.7]{CV} and \cite[Proposition
    8.1.10]{MR2724673}, we derive that $\reg(\calA_{d,\bdalpha})
    =\reg(\KK[\bdT]/J) =\reg(\KK[\bdT]/\ini(J)) =\pd ((\ini(J))^\vee)$. The
    formulas then follow from Propositions \ref{prop:reg_middle_part} and
    \ref{prop:reg_two_extremal_case}.
\end{proof}

Recall that the \emph{$\mathtt{a}$-invariant} of an algebra was
introduced by Goto and
Watanabe in \cite[Definition 3.1.4]{MR494707}. Since the Veronese type algebra $\calA_{d,\bdalpha}$
is Cohen--Macaulay by \Cref{thm:definingIdeals}, we have
\begin{equation}
    \mathtt{a}(\calA_{d,\bdalpha})=\reg(\calA_{d,\bdalpha})-\dim(\calA_{d,\bdalpha})
    \label{def:a-inv}
\end{equation}
in view of the equivalent definition of Castelnuovo--Mumford regularity in
\cite[Definitions 1 and 3]{MR676563}. 
Notice that the dimension and the Castelnuovo--Mumford regularity of $\calA_{d,\bdalpha}$ are known by \Cref{prop:analytic_spread} and \Cref{regSUM} respectively. Therefore, we obtain the $\mathtt{a}$-invariant of this algebra for free.
Moreover, the \emph{reduction number} of the ideal $I_{d,\bdalpha}$, and the
Castelnuovo--Mumford regularity of the algebra $\calA_{d,\bdalpha}$ are equal
by the following lemma. For the definition and further discussion of the
reduction numbers of ideals, see \cite[Section 8.2]{MR2266432}.

\begin{Lemma}
    [{\cite[Proposition 6.6]{CNPY} or \cite[Proposition 1.2]{MR3864202}}]
    \label{CNPY:6.6}
    Let $I$ be an equigenerated monomial ideal in some polynomial ring over a field $\KK$. Assume that the algebra  $\KK[I]$ is Cohen--Macaulay and the field $\KK$ is infinite. Then $I$ has the reduction number $\mathtt{r}(I) = \reg(\KK[I])$. 
\end{Lemma}

\begin{Corollary}
    Let $\calA_{d,\bdalpha}=\KK[I_{d,\bdalpha}]$ be the Veronese type algebra
    in \Cref{set:ideal_veronese}. Assume that $\KK$ is an infinite field and
    set $d'=\min(d,|\bdalpha|-d)$. Then, $ \mathtt{r}(I)=
    \floor{n-\frac{n}{d'}}$ and $\mathtt{a}(\calA_{d,\bdalpha})=
    -\ceil{\frac{n}{d'}}$.
\end{Corollary}
\begin{proof}
    The algebra $\calA_{d,\bdalpha}$ is Cohen--Macaulay by \Cref{thm:definingIdeals}. Thus, the statements follow from \Cref{prop:analytic_spread}, \Cref{regSUM}, \Cref{CNPY:6.6}, and \Cref{def:a-inv}.
\end{proof}

\begin{Remark}
    The $\mathtt{a}$-invariant of the Rees algebra $R[I_{d,\bdalpha}t]$ of a Veronese type ideal $I_{d,\bdalpha}$ is known in \cite[Corollary 12.6.6]{MR3362802} when $2\leq d\leq n$.  Its proof builds on the normality of the Rees algebra. Since the $\mathtt{a}$-invariant of the squarefree Veronese case is known, and equals the $\mathtt{a}$-invariant of the Veronese case, one can use the ``squeeze theorem'' method to obtain the $\mathtt{a}$-invariant the Rees algebra $R[I_{d,\bdalpha}t]$ of a general Veronese type ideal $I_{d,\bdalpha}$.

    However, this approach does not apply when we deal with the regularity (or equivalently, the $\mathtt{a}$-invariant) of the Veronese type algebra $\calA_{d,\bdalpha}$, which is the special fiber ring of the Veronese type ideal $I_{d,\bdalpha}$. Although the algebra $\calA_{d,\bdalpha}$ is still normal, the regularity of the $d$-th squarefree Veronese subring and the regularity of the $d$-th Veronese subring are different in general. For example, we can take $n=10$, $d=8$, and $\bdalpha=(1,\ldots,1,2,2)$ with $|\bdalpha|=12$. According to our formula, the regularity of $\calA_{d,\bdalpha}$ is $7$, the regularity of the $d$-th squarefree Veronese subring is 5, and the regularity of the $d$-th Veronese subring is $8$.
\end{Remark}

\subsection{Multiplicity bound of the algebra}
We conclude this work with a reasonable upper bound  on the multiplicity of the
Veronese type algebra $\calA_{d,\bdalpha}$. To begin with, we count the number
of generators of the ideal $I_{d,\bdalpha}$, the number of different equivalent
classes of maximal cliques, and the number of equivalent classes that have
precisely $(n-1)!$ maximal cliques.

\begin{Lemma}
    \label{Number_of_maximal_cliques} 
    The dimension of the polynomial ring $\KK[\bdT]\coloneqq \KK[T_\bda:\bda\in V_{n,d}^{\bdalpha}]$ is given by
    \[
        \dim (\mathbb{K}[\bdT])=\sum_{i\geq 0}
        (-1)^i\sum_{\substack{P\subseteq [n],\\ \#P=i}}\binom{d-\sum_{p\in P}(\alpha_{p}+1)+n-1}{n-1}.
    \]
    Moreover, there are
    \begin{align*}
        G\coloneqq \sum_{i, j\ge 0} (-1)^{i+j}
        \sum_{\substack{P\subseteq \{2,3,\dots,n-1\},\\ \#P=i}}
        \sum_{\substack{Q\subseteq \{1,n\},\\ \#Q=j}}\binom{d-\sum_{p\in P}(\alpha_{p}+1)-\sum_{q\in Q}\alpha_{q}+n-2}{n-1}
    \end{align*}
    different equivalent classes of maximal cliques in the graph $\calG(d,\bdalpha)$, and there are
    \begin{align*}
        H \coloneqq \sum_{i, j\ge 0} (-1)^{i+j}
        \sum_{\substack{P\subseteq \{2,3,\dots,n-1\},\\ \#P=i}}
        \sum_{\substack{Q\subseteq \{1,n\},\\ \#Q=j}}
        \binom{d-\sum_{p\in P}(\alpha_{p}-1)-\sum_{q\in Q} \alpha_{q}}{n-1}
    \end{align*}
    equivalent classes which have precisely $(n-1)!$ maximal cliques. In particular, there exists at most $H\cdot (n-1)!+(G-H)\cdot (n-1)!/2$ different maximal cliques.
\end{Lemma}

\begin{proof}
    It is clear that $\dim(\KK[\bdT])=\# V_{n,d}^{\bdalpha}$ is equal to the number of ways to write $a_1+\cdots+a_n=d$ such that $0\le a_i \le \alpha_i$.
    Furthermore, by \Cref{cor:Delta_maximal_clique} and \Cref{lem:legal_initial_tuple}, the number $G$ of equivalent classes of maximal cliques is the number of ways to have $a_1+\cdots+a_n=d$, under the conditions that $0\le a_i \le \alpha_i$ for $i=2,\ldots,n-1$, $1 \le a_1 \le \alpha_1$, and $0\le a_n \le \alpha_n-1$. It is not hard to see that $G = \# V_{n,d-1}^{\bdalpha'}$, where $\bdalpha'=(\alpha_1-1,\alpha_2,\alpha_3,\dots,\alpha_{n-1},\alpha_n-1)$.
    Similarly, by \Cref{prop:full_permutation}, finding the number $H$ is counting how many possible ways one can write $a_1+\cdots+a_n=d$, subject to the conditions $1 \le a_i \le \alpha_i-1$ for $i=2,\ldots,n-1$,  $1\le a_1 \le \alpha_1$, and $0 \le a_n \le \alpha_n-1$. It is not difficult to see that $H = \# V_{n,d-(n-1)}^{\bdalpha''}$, where $\bdalpha''=(\alpha_1-1,\alpha_2-2,\alpha_3-2,\dots,\alpha_{n-1}-2,\alpha_n-1)$.
    The three formulas given above then follow from the classical \emph{stars and bars method} and the \emph{inclusion-exclusion principle} in combinatorics.

    Notice that if an equivalence class $\calE$ does not have $(n-1)!$ maximal cliques, then its poset of obstructions is not trivial by \Cref{prop:full_permutation}. Say, we have $p \,\triangleleft\, q$ in this poset. Then, for any $A\in \calE$, we have $p$ preceding $q$ in $\sgn(A)$. Thus,
    \[
        \# \calE = \# \Set{\sgn(A):A\in \calE} \le (n-1)!/2,
    \]
    namely, $\calE$ contains at most $(n-1)!/2$ maximal cliques. The ``in particular'' part then follows.
\end{proof}

For a homogeneous $\mathbb{K}$-algebra $R$, let $\mathtt{e}(R)$ denote its \emph{multiplicity} 
with respect to its graded maximal ideal.
This number is clear in the squarefree case, by the work of Terai \cite{MR1715588}.

\begin{Lemma}
    [{\cite[Lemma 4.1]{MR1715588}}]
    \label{lem:Terai_lem}
    Let $\fraka$ be a
    {squarefree} monomial ideal in a polynomial ring $A$. Then
    $\mathtt{e}(A/\fraka)=\beta_{1,h_1}(A/\fraka^\vee)$, where
    $h_1=\indeg(\fraka^\vee)$ is the initial degree of the Alexander dual
    ideal $\fraka^{\vee}$.
\end{Lemma}

\begin{Corollary}
    \label{cor:mult_is_number_of_max_clique}
    The multiplicity of $\calA_{d,\bdalpha}$ is equal to
    $\#\MC(\calG(d,\bdalpha))$.
\end{Corollary}
\begin{proof}
    Note that $\mathtt{e}(\calA_{d,\bdalpha})$ can be calculated from the
    Hilbert series of $\calA_{d,\bdalpha}\cong \KK[\bdT]/J$. Meanwhile, the
    Hilbert series of $\KK[\bdT]/J$ and $\KK[\bdT]/\ini(J)$ coincide. Thus,
    thanks to \Cref{lem:Terai_lem}, in order to find
    $\mathtt{e}(\calA_{d,\bdalpha})$, we only need to compute the minimal
    number of generators of the equigenerated squarefree monomial ideal
    $(\ini(J))^\vee$. This number is obviously the number of maximal cliques of
    $\calG$.
\end{proof}

In addition, Terai gave the following upper bound on multiplicity.

\begin{Lemma}
    [{\cite[Theorem 4.2]{MR1715588}}]
    \label{TeraiBound}
    Let $R=A/\fraka$ be a homogeneous $\mathbb{K}$-algebra of codimension $g \geq
    2$. Then
    \[
        \mathtt{e}(R)\le \binom{\reg(R)+g}{g}-\binom{\reg(R)-\indeg(\fraka)+g}{g}.
    \]
\end{Lemma}

Relatedly, Eisenbud and Goto in \cite{MR0741934} made a conjecture linking
Castelnuovo-Mumford regularity and multiplicity. Although a counterexample was
recently given by McCullough and Peeva in \cite{MR3758150}, the statement of
the original conjecture still holds in the Cohen-Macaulay case. 

\begin{Lemma}
    [{\cite[Corollary 4.15]{MR2103875}}]
    \label{EG}
    Suppose that $A$ is a polynomial ring over an algebraically closed field.
    If $\fraka$ is a nondegenerated homogeneous prime ideal in $A$ and
    $A/\fraka$ is Cohen--Macaulay, then
    \[
        \reg(A/\fraka) \le \mathtt{e}(A/\fraka)-\codim(\fraka).
    \]
\end{Lemma}

We are ready to state our final result regarding the multiplicity of $\calA_{d,\bdalpha}$.

\begin{Theorem}
    \label{MultiBound}
    Suppose that $\calA_{d,\bdalpha}=\KK[I_{d,\bdalpha}]$  is the Veronese type
    algebra of \Cref{set:ideal_veronese} whose presentation ideal $J$ is not
    zero. Let $r\coloneqq \reg( \calA_{d,\bdalpha})$, $t\coloneqq
    \dim(\KK[\bdT])$, $G$ and $H$ be the numbers that we computed in
    \Cref{regSUM} and \Cref{Number_of_maximal_cliques}. Furthermore, we write
    $d_1\coloneqq d$ and $d_2\coloneqq \sum_{i=1}^n\alpha_i-d$.
    Then, we have
    \begin{align*}
        r+t-n \le \mathtt{e}(\calA_{d,\bdalpha}) &\le \min\Bigg\{ H\cdot(n-1)!+\frac{(G-H)\cdot(n-1)!}{2},\; d_1^{n-1},\; d_2^{n-1},
            \; \\
            & \qquad \qquad \qquad  \binom{r+t-n}{t-n}-\binom{r-2+t-n}{t-n}
        \Bigg\}.
    \end{align*}
\end{Theorem}

\begin{proof}
    The codimension of $J$ is $t-n$ by \Cref{prop:analytic_spread}. Since
    we can replace $\KK$ by its algebraic closure, the first inequality
    follows from \Cref{EG}.  
    As for the second inequality, we apply first Lemmas
    \ref{Number_of_maximal_cliques} and \ref{cor:mult_is_number_of_max_clique}.
    In addition, notice that the multiplicity in the Veronese case is
    well-known: $\mathtt{e}(\calA_{d_1,(d_1,\ldots,d_1)})=d_1^n$. Since
    obviously $\MC(\calG(d,\bdalpha))\subseteq
    \MC(\calG(d_1,(d_1,\dots,d_1)))$, we have
    $\mathtt{e}(\calA_{d,\bdalpha})\le d_1^{n-1}$ by
    \Cref{cor:mult_is_number_of_max_clique}.  Similarly, since
    $\calA_{d,\bdalpha}\cong\calA_{d_2,\bdalpha}$, we also have
    $\mathtt{e}(\calA_{d,\bdalpha})\le d_2^{n-1}$. 
    As for the remaining piece, we observe that the presentation ideal $J$
    and its initial ideal are quadratic by \Cref{thm:definingIdeals}. Thus,
    when $\codim(J)\ge 2$, we can apply \Cref{TeraiBound}. If instead
    $\codim(J)=1$, since $\ini(J)$ is height-unmixed, squarefree, and has
    the same height, this initial ideal is the intersection of principal
    monomial prime ideals. Consequently, the quadratic ideal $\ini(J)$ is
    generated by a squarefree monomial of degree $2$.
    In particular, $\mathtt{e}(\calA_{d,\bdalpha})
    =\mathtt{e}(\KK[\bdT]/\ini(J))=2$. On the other hand, 
    \[
        \binom{r+t-n}{t-n}-\binom{r-2+t-n}{t-n}=
        \binom{r+1}{1}-\binom{r-2+1}{1}=2,
    \]
    which means that we have equality in this case.
\end{proof}

\begin{Example}
    Let $\bdalpha=(1,4,4,5,7)$ and $d=7$. Using the notation in
    \Cref{MultiBound}, we have that $d_2^4=(21-7)^4=14^4=38416>d_1^4=7^4=2401$.
    By \Cref{Number_of_maximal_cliques}, we have $t=171$, $G=75$, and $H=18$.
    Therefore $H\cdot (5-1)!+(G-H)\cdot (5-1)!/2=1116$.  By \Cref{regSUM}, we
    have $r=\reg(\calF(\Sss(\bdeta)))=5-1=4$. Thus, $r+t-n=4+171-5=170$ and
    \begin{align*}
        \binom{r+t-n}{t-n}-\binom{r-2+t-n}{t-n}
        & = \binom{4+171-5}{171-5}-\binom{4-2+171-5}{171-5}\\
        &=33571342.
    \end{align*}
    It follows from \Cref{MultiBound} that we can obtain
    \[
        170\le \mathtt{e}(\calA_{d,\bdalpha})\le 1116.
    \]
    By directly enumerating the maximal cliques, we can check that
    $\mathtt{e}(\calA_{d,\bdalpha})=960$ by \Cref{lem:Terai_lem}.
\end{Example}

\begin{acknowledgment*}
    It is our pleasure to thank Rafael Villarreal for helpful suggestions. The authors sincerely thank the patient reviewer for the very helpful
and constructive suggestions that greatly improve the presentation of this manuscript.
    We are also grateful to the software system \texttt{Macaulay2} \cite{M2}, for serving as an excellent source of inspiration. The second author is partially supported by the ``Anhui Initiative in Quantum Information Technologies'' (No.~AHY150200) and the ``Innovation Program for Quantum Science and Technology'' (2021ZD0302902).
\end{acknowledgment*}

\bibliography{SymmetricBib}
\end{document}